\documentclass{article}
\usepackage[utf8]{inputenc}
\usepackage[colorlinks=true, urlcolor=blue,citecolor=blue,anchorcolor=blue]{hyperref}

\usepackage{fullpage}

\usepackage{dsfont,amsthm, amssymb,amsmath,stmaryrd,authblk,mathtools,tikz,graphicx}
\usetikzlibrary{positioning}
\usetikzlibrary{decorations.pathmorphing}
\usepackage{caption,subcaption}

\tikzset{
    dot/.style 2 args={fill, circle, inner sep=1pt, label={#1:\scriptsize #2}}
}

\usepackage{graphicx,color,tikz}
\usepackage{relsize}
\usepackage{url}
\usepackage{hyperref}
\usepackage[capitalise]{cleveref}
\usepackage{mymacros}

\usepackage{breqn}

\allowdisplaybreaks
\newcommand\numberthis{\addtocounter{equation}{1}\tag{\theequation}}

\makeatletter
\let\cref@old@eq@setnumber\eq@setnumber
\def\eq@setnumber{%
\cref@old@eq@setnumber%
\cref@constructprefix{equation}{\cref@result}%
\protected@xdef\cref@currentlabel{%
[equation][\arabic{equation}][\cref@result]\p@equation\theequation}}
\makeatother

\title{Hamiltonian surgery: Cheeger-type inequalities for nonpositive (stoquastic), real, and Hermitian matrices}
\author[1]{Michael Jarret \thanks{mjarret@pitp.ca}}
\affil[1]{\footnotesize\PITP}
\date{\today}

\begin{document}

\maketitle

\begin{abstract}
    Cheeger inequalities bound the spectral gap $\gamma$ of a space by isoperimetric properties of that space and vice versa. In this paper, I derive Cheeger-type inequalities for nonpositive matrices (aka stoquastic Hamiltonians), real matrices, and Hermitian matrices. For matrices written $H = L+W$, where $L$ is either a combinatorial or normalized graph Laplacian, I show that,
    \begin{enumerate}
        \item when $W$ is diagonal and $L$ has maximum degree $d_{\max}$, $2h \geq \gamma \geq \sqrt{h^2 + d_{\max}^2}-d_\max$;
        \item when $W$ is real, we can often route negative-weighted edges along positive-weighted edges such that the Cheeger constant of the resulting graph obeys an inequality similar to that above; and
        \item when $W$ is Hermitian, the weighted Cheeger constant obeys $2h \geq \gamma$ 
    \end{enumerate}
    where $h$ is the weighted Cheeger constant of $H$. This constant reduces bounds on $\gamma$ to information contained in the underlying graph and the Hamiltonian's ground-state. 
    
    If efficiently computable, the constant opens up a very clear path towards adaptive quantum adiabatic algorithms, those that adjust the adiabatic path based on spectral structure. I sketch a bashful adiabatic algorithm that aborts the adiabatic process early, uses the resulting state to approximate the weighted Cheeger constant, and restarts the process using the updated information. Should this approach work, it would provide more rigorous foundations for adiabatic quantum computing without \textit{a priori} knowledge of the spectral gap. 
\end{abstract}

\section{Introduction}

\subsection{Motivation}
An $n \times n$ Hermitian matrix $H$ has eigenvalues $\lambda_0 \leq \lambda_1 \leq \dots \leq \lambda_{n-1}$. We call the difference in the two lowest eigenvalues of $H$, $\gamma = \lambda_1 - \lambda_0$, its spectral gap. Bounding the spectral gap is a problem that could be motivated any number of ways. In quantum theory, the spectral gap determines the runtime of adiabatic algorithms and processes \cite{Jansen2006,Albash2018,Crosson2016} and relates to quantum phase transitions \cite{sachdev2007quantum}. The spectral gap is also intimately related to the rate at which heat diffuses on a manifold \cite{yau2009estimate,andrews2011proof} and the rate at which substochastic processes approach their quasistationary distributions \cite{collet2012quasi,collet2013markov}. At the computational level, it determines the runtime of various well-known randomized algorithms \cite{sinclair2012algorithms} as well as Fleming-Viot type algorithms for approximating marginals \cite{Jarret2016,jarret2017substochastic,cloez2016quantitative,cloez2016fleming}. Each of these is an independently interesting topic, which would motivate its own study of the spectral gap. 

Here, I abstract away the context and seek to understand the spectral structure of $H$ by decomposing it as $H=L+W$, the sum of a graph Laplacian $L$ and some other Hermitian matrix $W$. All Hermitian matrices can be decomposed this way and, as we will see, the decomposition proves fruitful. If $W$ is diagonal, $H$ is frequently called a ``stoquastic'' Hamiltonian or ``stoquastic'' matrix. A diagonal $W$ also implies that $H$ is an infinitesimal generator of a substochastic process and the resulting matrix $I-\epsilon H$ is a substochastic matrix. When $W$ is not diagonal, but instead real, the matrix $H$ may have a \textit{sign problem}, or all off-diagonal terms may not have the same sign. The ``problem'' is that such Hamiltonians can be difficult to study with Monte Carlo methods \cite{troyer2005computational}. Finally, when $W$ is a general Hermitian matrix, then $H$ has no special name; Hamiltonian is special enough. 

In this paper, I look to formalize the relationship between $\gamma$ and some geometrical properties of the ground-state $\phi_0$ of $H$, or its lowest eigenvector. I always assume that $H$ is represented in such a way that $H:\mathbb{C}^{\abs{V}} \longrightarrow \mathbb{C}^{\abs{V}}$ for some graph $G=(V,E)$. In our representation, $L$ is the graph Laplacian of $G$. Correspondingly, we consider functions $\phi:V \longrightarrow \mathbb{C}$. I often assume that $H$ has been rotated by a diagonal unitary transformation such that $\phi_0 \geq 0$ and will define a weighted Cheeger constant $h$ \cite{Chung2000}, capturing the relevant geometric properties of $\phi_0$. It remains unclear how difficult approximating $h$ is, however in the event that $W=0$, it reduces to the Cheeger constant of $G$ and can be efficiently approximated. If and when one can approximate $h$ remains a very important open question beyond the scope of this paper, though I discuss some related ideas in \cref{sec:discussion}.

The conceptual lesson of this paper is quite concrete. For any Hermitian matrix $H$, if $H$ has a large spectral gap, then $\phi_0$ \textit{has no bottlenecks}. That is to say, that $\phi_0$ is a somewhat smooth distribution over $G$. Prior results, discussed below, suggest that we should already believe this, but leave open the possibility that there exist cases that betray our intuition. Provided that $H$ is not diagonal, I show that our intuition is always correct. (In the case that $H$ is diagonal, our intuition is trivially correct.) I do not, however, show the converse. That is, I leave open the question of whether a small spectral gap implies a bottlenecked $\phi_0$. I show that this is indeed implied in the stoquastic and some real cases, but when this is implied by the general Hermitian case is left open. Adapting these techniques to more general cases appears possible and I will discuss some potential approaches as we progress through the proof. Furthermore, in \cref{sec:discussion}, we will see that understanding the precise relationship might have far-reaching implications for quantum adiabatic algorithms.

\subsection{Previous Work}\label{sec:previous}
In this paper, we study isoperimetric inequalities of discrete systems. Such inequalities enjoy a rich history. Within the context of randomized algorithms, the Cheeger constant often provides a means of determining the mixing time of a Markov chain and, thus, the efficiency of certain approximation algorithms \cite{sinclair2012algorithms}. Standard Cheeger inequalities relate the spectral gap $\gamma$ of the Laplacian $L$ corresponding to a graph $G$ and the Cheeger constant $h$ of that graph. They usually appear in a form similar to
\begin{equation}\label{eqn:standard}
    2h \geq \gamma \geq \frac{h^2}{2}
\end{equation}
and provide a very useful, intuitive significance to the spectral gap. Although a useful quantity, we know that computing the Cheeger constant exactly for an arbitrary graph is NP-hard \cite{GAREY1976237,leighton1988approximate,kaibel2004expansion}. Despite this hardness, the Cheeger constant can indeed be efficiently approximated \cite{sinclair2012algorithms,kannan2004clusterings}. 

The spectral gap, and hence Cheeger constant, is also of primary interest in spectral graph theory, where it is often explored in connection with graph Laplacians \cite{Chung}. In \cite{Chung2000}, the authors adapted Cheeger inequalities to apply to the gap in the \textit{Dirichlet eigenvalues} of a graph. The distinguishing characteristic of the Dirichlet eigenvalues is that they arise by imposing a Dirichlet boundary constraint. This constraint requires that, for some subset of vertices $\delta V \subseteq V$, all eigenfunctions must satisfy $f\vert_{\delta V} = 0$. These eigenvalues are also studied quite a bit and numerous bounds appear in the literature. Unfortunately for us, these studies typically focus on the easier problem of bounding eigenvalues, not their differences. Additionally, the few gap inequalities that exist, like those in \cite{Chung2000}, are not easily applied to most situations we are presently interested in. Thus, we require a new inequality.

To this end, various authors (including me) have pursued Cheeger-type inequalities in the stoquastic case \cite{al2010energy,Jarret2014a} and more general Hermitian matrices \cite{crosson2017quantum}. In either case, this problem is actually equivalent to that of determining the differences in the Dirichlet eigenvalues of an appropriate host graph. These inequalities all assume an unfortunate form that looks something like
\begin{equation}\label{eqn:old}
    2 \norm{H} h \geq \gamma \geq \frac{h^2}{2\norm{H}}
\end{equation}
where $h$ is an appropriately defined Cheeger constant. We can easily see the weakness of this expression: unlike in the case of graph Laplacians, it is entirely possible that $h^2 \sim \norm{H} \sim e^{n}$. Thus, the lower bound from \cref{eqn:old} scales like a constant, whereas we would expect from \cref{eqn:standard} that $\gamma \gtrsim e^{n}$. A similar argument illuminates the weakness of the upper bound. Suppose the very common situation that $\norm{H} \sim e^{n}$ and $h\sim e^{-n}$. Then, the upper bound on $\gamma$ scales as a constant whereas we expect that $\gamma \lesssim e^{-n}$. This latter issue leaves open the possibility that one might have a large spectral gap in the presence of a bottleneck. In this work, I will correct these defects.

\subsection{Results}
Consider a graph $G=(V,E)$ with edge weights assigned by $w:V\times V \longrightarrow \mathbb{R}$. Then, for the corresponding graph Laplacian $L$ and any real diagonal matrix $W$, $H=L+W$ admits a weighted Cheeger constant $h$, defined in \cite{Chung2000} and again in \cref{sec:cheeger}. In particular, I prove that for any stoquastic matrix with spectral gap $\gamma$
\begin{equation}\label{eqn:result}
    {2 h \geq \gamma \geq \sqrt{h^2 + Q^2} - Q}
\end{equation}
where, if $L$ is a combinatorial Laplacian, $Q$ is the maximum degree of a vertex of $G$. If $L$ is a normalized Laplacian, $Q=1$. 

For any real matrix, we can identify positive off-diagonal terms with negative edge weights ($E^- = \{\{u,v\} \in E \vert \allowbreak w(u,v) \allowbreak < 0 \allowbreak \}$) and show that
\[
    {2 h \geq \gamma \geq \sqrt{k^2 + Q^2} - Q}.
\]
if $\phi_0$ is uniform up to phase and 
\[
    {2h \geq \gamma \geq (Q+\rho)^2 - \sqrt{(Q+\rho)^2 -k^2}}
\]
where $\rho = \lambda_{\abs{V}-1}-\lambda_0$ otherwise. Above, $h$ is the weighted Cheeger constant corresponding to the graph $G^+ = (V,E\setminus E^-)$ under the original weight function and $k$ the weighted Cheeger constant of $G^+$ with a redistributed weight function $w^+$ to be defined in \cref{sec:neg_edges}. In \cref{sec:applications}, we will see that these equations can often be relaxed to 
\[
    {2 h \geq \gamma \geq \epsilon
    \left(\sqrt{h^2+Q^2}-Q\right)}
\]
for a constant $\epsilon$, which may be easier to apply and retains appropriate scaling behavior. In other words, at least asymptotically, I reduce the problem of bounding the gap of a signed graph $G$ to that of determining the appropriate Cheeger constant of $G^+$.

Finally, I provide the upper bound
\begin{equation}\label{eqn:result2}
    2h \geq \gamma
\end{equation}
for any Hermitian matrix.

Not only does this expression correct the problems mentioned in \cref{sec:previous}, but the improvement over these statements can be quite drastic and firmly establishes some conceptual points. Note that in cases where $h$ is large compared to the maximum degree $Q$, which often happens when $\norm{W}$ is sufficiently large, the lower bound in \cref{eqn:old} becomes weak whereas \cref{eqn:result} remains tight. Furthermore, the form of the expression guarantees that the inequality scales appropriately for all relative sizes of $L$ and $W$ and, hence, all Hermitian matrices. Although establishing the lower bound in \cref{eqn:old} is unlikely in general, expanding around $h\approx 0$ does yield a similar expression. Furthermore, when $h$ is large relative to $Q$, \cref{eqn:result} guarantees that $\gamma \sim h$.

The efficiency with which one can classically approximate $h$ remains unclear, but the quantity only depends upon information about the ground-state distribution of $H$ and the corresponding graph $G$. This opens up the possibility that an adiabatic algorithm may be able to efficiently approximate $h$, even if a classical method remains elusive. This ability would be a great advantage to the field of adiabatic optimization, as it could be used to determine the appropriate time dependence of an adiabatic evolution without \textit{a priori} knowledge of the spectral gap. Such an evolution can be necessary to produce quantum speedups, like those achieved in adiabatic Grover search \cite{roland2002quantum}. This idea will be discussed in detail in \cref{sec:discussion}, but conclusive results, should they exist, are left for future work.

\section{Preliminaries}
\subsection{The Rayleigh Quotient}\label{sec:rayleigh}
For an $n\times n$ Hermitian operator $H$ acting on the space $\mathcal{S} = \{f: \intrange{1}{n} \longrightarrow \mathbb{C}^n\}$, one defines the Rayleigh quotient corresponding to a function $f \in \mathcal{S}$ as
\begin{equation}\label{eqn:Rayleigh}
    R(H,f) = \frac{\langle f, H f\rangle}{\langle f, f\rangle}.
\end{equation}
Thus, the eigenvalues $\lambda_0(H) \leq \lambda_1(H) \leq \dots \leq \lambda_{n-1}(H)$ of $H$ can be written as
\begin{equation}\label{eqn:eigenvalues}
    \lambda_i(H) = \inf_{f \perp T_{i-1}}\frac{\langle f, H f\rangle}{\langle f, f\rangle}
\end{equation}
where $T_{i}$ is the space spanned by the functions $f_j$ achieving $\lambda_j(H)$ for each $0 \leq j \leq i$. We call $f$ achieving $\lambda_0(H)$ the \textit{ground-state}. Of particular interest in this paper is the spectral gap $\gamma(H)$ of $H$, or the difference in its two lowest eigenvalues, $\gamma(H) = \lambda_1(H) - \lambda_0(H)$. Usually, we will just write $\gamma = \gamma(H)$ and $\lambda_i = \lambda_i(H)$ and reserve the argument for when it is necessary to distinguish the eigenvalues of two matrices.

Our first goal is to rewrite $H$ in a form useful for the current work. Presently, we only seek lower bounds for real matrices, so we can prove a quick comparison theorem between $\gamma(H)$ and $\gamma(\Re(H))$ where $\Re(H)$ is the real part of $H$. One can immediately obtain a useful upper bound on the spectral gap of $H$ by considering the function $\phi_0$ obtaining $\lambda_0(H)$ in \cref{eqn:eigenvalues}. 
\begin{prop}\label{prop:Hermitian}
For a Hermitian matrix $H$ with spectral gap $\gamma(H)$, ground-state $\phi_0$, and $U= \mathrm{diag}(\phi_0/\abs{\phi_0})$ where the ratio and absolute value are taken pointwise,
\[
    \gamma(H) \leq \gamma(\Re(U^\dagger H U)).
\]
\end{prop}
\begin{proof}
This proof is very straightforward. First, suppose $H$ has ground-state $\phi_0$. Then, let $U = \mathrm{diag}(\phi_0/\abs{\phi_0})$ where the ratio and absolute value are taken pointwise. Obviously, $U$ is unitary and $U^\dagger \phi_0 \geq 0$. Now, write $\Im(U^\dagger H U) = i S$, where $S \in \mathbb{R}^{n \times n}$ is skew-symmetric. Thus, $\lambda_0$ satisfies
\begin{align*}
    \lambda_0 &=\inf_{f \in \mathbb{C}^n}\frac{\langle f,Hf\rangle}{\langle f,f\rangle}\\
    &=\inf_{U^\dagger f \in \mathbb{C}^n}\frac{\langle U f,H U f\rangle}{\langle U f,U f\rangle}\\
    &=\inf_{f \geq 0}\frac{\langle f,U^\dagger H U f\rangle}{\langle f,f\rangle}\\
    &=\inf_{f \geq 0}\frac{\langle f,\Re(U^\dagger H U)f\rangle+\langle f,iS f\rangle}{\langle f,f\rangle}\\
    &=\inf_{f \geq 0}\frac{\langle f,\Re(U^\dagger H U)f\rangle}{\langle f,f\rangle}
\end{align*}
where the second equality follows from our choice of $U$ and the final equality from the skew-symmetry of $S$. Now, the Rayleigh quotient for $\lambda_1$ becomes 
\begin{align*}
    \lambda_1 &= \inf_{\substack{f \perp \phi_0 \\ f \in \mathbb{C}^n}}\frac{\langle f , H f\rangle}{\langle f , f\rangle}\\
    &= \inf_{\substack{U f \perp \phi_0 \\ f \in \mathbb{C}^n}}\frac{\langle U f , H U f\rangle}{\langle U f , U f\rangle}\\
    &= \inf_{\substack{f \perp U^\dagger \phi_0 \\ f \in \mathbb{C}^n}}\frac{\langle f , U^\dagger H U f\rangle}{\langle f , f\rangle}\\
    &= \inf_{\substack{f \perp U^\dagger \phi_0 \\ f \in \mathbb{C}^n}}\frac{\langle f , \Re(U^\dagger H U) f\rangle + \langle f, iS f \rangle}{\langle f , f\rangle}\\
    &\leq \inf_{\substack{f \perp U^\dagger \phi_0 \\ f \in \mathbb{R}^n}}\frac{\langle f , \Re(U^\dagger H U) f\rangle + \langle f, iS f \rangle}{\langle f , f\rangle}\\
    &=\inf_{\substack{f \perp U^\dagger\phi_0 \\ f \in \mathbb{R}^n}}\frac{\langle f , \Re(U^\dagger H U) f\rangle}{\langle f , f\rangle}.
\end{align*}
Above, the inequality follows from introducing the additional constraint on the infimum. Thus, the gap of $\gamma(H) \leq \gamma(\Re(U^\dagger H U))$.
\end{proof}

\subsection{Stoquastic Hamiltonians}
\Cref{prop:Hermitian} guarantees us that, at least in the case of upper bounds, we hereafter need only consider $\Re(U^\dagger H U)$. Hence, we no longer address the issue of upper bounding the gap of a Hermitian matrix, since the bound is implied by any bounds on real matrices. Although determining an appropriate $U$ to actually perform the rotation in \cref{prop:Hermitian} might be a hard problem in general,\footnote{I would conjecture that, since the problem of determining a graph's frustration index is NP-hard \cite{sher,barahona} actually determining this unitary should be NP-hard. That one can efficiently detect whether a signed graph is balanced implies, with only slight modification, that one can efficiently detect whether a Hamiltonian is stoquastic \cite{HARARY1980131} in this simple case. Finding a unitary which makes a general Hamiltonian stoquastic is NP-complete \cite{Marvian2018}.} there exist certain cases where this becomes relatively easy. One convenient way to describe these situations is through the \textit{frustration index} of the matrix $\Theta = H/\abs{H}$ where, again, the ratio and absolute value are taken pointwise. 

If we view $\Theta$ as an adjacency matrix, as will be made precise in the following section, we can consider a cycle cover of $\Theta$ given by the successor function $\sigma:\intrange{1}{n} \longrightarrow \intrange{1}{n}$. Here, $\sigma$ is just a permutation of $\intrange{1}{n}$. Then, the sequence $i \rightarrow \sigma(i) \rightarrow \sigma \cdot \sigma(i)\rightarrow \dots \rightarrow i$ is a cycle through $\Theta$, which we refer to as $c_\sigma(i)$.  We call the set of all successor functions $C = \{c_\sigma\}$. 

For any $1\leq i\leq n$, we define the signature of the cycle $c_\sigma(i)$ as
\[
    {\mathrm{sig}(c_\sigma(i)) = \prod_{k \in c_\sigma(i)} \left[-\Theta_{k,\sigma(k)}\right] = (-1)^{\abs{c_\sigma(i)}} \prod_{k \in c_\sigma(i)} \Theta_{k,\sigma(k)}.}
\]
In analogy to the standard definition, we somewhat carelessly define the \textit{frustration index} of $\Theta$ as the minimum number of elements of $\Theta$ that need to be removed such that $\mathrm{sig}(c_\sigma(i)) \in \{0,1\}$ for all $c_\sigma \in C$ and $i \in \intrange{1}{n}$ \cite{Atay2014,Martin2017,Lange2015}. This particular definition is clearly far from ideal, since complex phases imply that this is not a strict question of combinatorics, and we should prefer a functional definition similar to that of \cite{Lange2015} in the future. Despite its failings, we can use this definition to define \textit{stoquastic} matrices.
\begin{define}
    We call a matrix \textit{stoquastic} if it has frustration index $0$.
\end{define}
This definition of stoquastic diverges from much of the literature on the subject. (See, e.g. \cite{bravyi2008complexity}.) Nonetheless, it is a bit more descriptive and (potentially) avoids redefining well-known mathematical concepts.\footnote{Stoquastic Hamiltonians are typically one in a long list of names for matrices with nonpositive (or nonnegative) off-diagonal terms. Nonetheless, I would be incredibly surprised if the extended class here has escaped a pre-existing label.} We introduce this definition for two reasons: (1) because frustration index has been used to obtain better isoperimetric inequalities \cite{Martin2017,Lange2015}, setting the stage for future work, and (2) because it extends our results to a broader class of matrices. Importantly, this property can be efficiently checked (at least in the dimension of the matrix), so that one can determine whether or not stoquastic spectral bounds apply even if one is unsure that a matrix is stoquastic. Thus, this definition makes the methods presented below easier to apply in many cases.  

The unitary $U$ that transforms $H$ such that all off-diagonal terms of $U^\dagger H U$ are nonpositive is immediate. First for any cycle we can decompose $\sigma$ into paths $\sigma_1$ and $\sigma_2$.
\begin{align*}
    1 &=\prod_{k\in c_\sigma(i)} \left[-\Theta_{k ,\sigma(k)}\right]\\ &= \left(\prod_{k\in c_{\sigma_1}(i)} \left[-\Theta_{k ,\sigma(k)}\right]\right)\left( \prod_{k \in c_{\sigma_2}(i)}\left[-\Theta_{k,\sigma(k)}\right]\right) \\
    &=\left(\prod_{k\in c_{\sigma_1}(i)} \left[-\Theta_{k ,\sigma_1(k)}\right]\right)\left( \prod_{k\in c_{ \sigma_2}^{-1}(i)}\left[-\Theta^\dagger_{k,\sigma_2(k)}\right]\right)
\end{align*}
where the final line follows because, since $\Theta$ is Hermitian, every point in a cycle forms its own cycle. In other words, beginning at $i$, the product $\prod_{k \in c_\sigma(i)}^{\sigma^j(i)}(-\Theta_{k,\sigma(k)})$ is entirely independent of the particular path chosen. This immediately implies the well-known fact that the frustration index of a real, nonnegative $\Theta$ is $0$ if and only if $\Theta$ describes a bipartite graph. The path-independence above also implies that one can explicitly construct the appropriate unitary $U$ by choosing a vertex, say $i$ and then, for all $j$ in some cycle with $i$, $U_{jj} = \prod_{k=i}^{\sigma^{-1}(j)}(-\Theta_{ij})U_{ii}$. Since every pair of vertices forms a simple cycle, this reduces to the constraint that, provided $\Theta_{ij}\neq 0$, $U_{jj}=-\Theta_{ij}U_{ii}$. Thus, we know that this definition of $U$ is consistent and unique up to a global phase. Furthermore, it clearly performs the appropriate transformation. Thus, if we satisfy stoquasticity, we know \textit{a priori} that $U^\dagger H U$ has all nonpositive off-diagonal elements. More importantly, because $U$ is diagonal, we do not need to do the unitary transformation; we can simply replace each off diagonal term $w_{uv}$ with $-\abs{w_{uv}}$ and obtain the resulting matrix.

Despite the utility of this condition in producing bounds for a larger class of matrices, in what follows we assume the problem has been reduced such that $H \mapsto U^\dagger H U$, guaranteeing that all off-diagonal terms are nonpositive and the ground-state $\phi_0 \geq 0$. This allows for a simpler presentation.

\subsection{Graph Laplacians}
We wish to characterize $\Re(U^\dagger H U)$ in terms of graph Laplacians. Although the standard combinatorial and normalized graph Laplacian are defined such that all diagonal elements are nonnegative and all off-diagonal elements are nonpositive, we can relax the latter constraint and consider \textit{signed} Laplacians. For our purposes, the only difference between a signed and standard Laplacian is that signed Laplacians have no constraint on the non-positivity of their off-diagonal terms, however our definitions are somewhat atypical \cite{atay2014spectrum}.\footnote{The signed Laplacian typically has the degree of vertex $u$ equal to the absolute value of the sum of the edge-weights incident on the vertex. Because we are about to allow for an arbitrary diagonal perturbation, we will also be able to recover the standard combinatorial signed Laplacian by taking $W_u \mapsto W_u + \sum_v ( \abs{w_{uv}} - w_{uv})$.}

We begin by considering a connected weighted graph $G=(V,E)$ with weight function $w:V\times V \longrightarrow \mathbb{R}$ where we require $w(u,v)=0$ whenever $(u,v) \notin E$. Additionally, we require that $w(u,v)= w(v,u)$ or that $G$ is undirected.\footnote{Although directed graphs do not correspond to the physical systems that we are presently interested in and are thus omitted, extending these results to such a setting is still  well-motivated. For some results on directed graphs, see, e.g. \cite{Chung2005,Bauer2012,Chan2015}.} For ease of presentation, we will also lower arguments to $w$ such that $w_{uv} = w(u,v)$. Since we are allowing the possibility of negative edge weights, we introduce the notation $E^+ = \left\{\{u,v\} \in E \vert w_{uv} > 0 \right\}$ for the set of all positive-weighted edges and $E^-$ for the set of all negative-weighted edges. We also define $G^\pm = (V,E^\pm)$ and note that $G^+\subseteq G$. Now we can include some standard definitions for the combinatorial and normalized Laplacians, keeping in mind that edge weights may be negative.

\subsubsection{The combinatorial Laplacian}
To define the combinatorial Laplacian for a graph $G$, we first let the degree of a vertex $u \in V$ be $d_u = \sum_{v}w_{uv}$. Then, the combinatorial graph Laplacian $L$ is
\[
    L(u,v) = \begin{cases}
        d_u & u=v \\
        -w_{uv} & u\neq v
    \end{cases}
\]
where $d_u = \sum_v w_{uv}$. 

For any function $f:V \longrightarrow \mathbb{R}$ (or $f:V \longrightarrow \mathbb{C}$), one can easily see that
\begin{equation*}
    Lf(u) = \sum_{v}w_{uv}[f(u)-f(v)]
\end{equation*}
where we have adopted the standard convention that $Lf(u) = [Lf](u)$. (This is just to say that $Lf \neq L \circ f$, since $L: \mathbb{R}^{\lvert V \rvert} \longrightarrow \mathbb{R}^{\lvert V \rvert}$.) One can easily argue that if $f$ is an eigenfunction of $L$, then $f$ satisfies
\begin{equation*}
    \lambda f(u) = \sum_v w_{uv}[f(u)-f(v)].
\end{equation*}
Now, let $W:V \longrightarrow \mathbb{R}$. We can represent $W$ as an $n \times n$ diagonal matrix and write $W_u \equiv W_{uu}$. Then, if $f$ is an eigenfunction of $L + W$, $f$ satisfies
\begin{equation}\label{eqn:combinatorial_operator}
    (\lambda-W_u) f(u) = \sum_v w_{uv}[f(u)-f(v)].
\end{equation}
Recalling the definition of the Rayleigh quotient, $R(L+W,f)$, we have that the eigenvalues of $L+W$ satisfy
\begin{equation}\label{eqn:rc-pert}
    \lambda_i = \inf_{\substack{\tiny{f \perp T_{i-1}}}} \frac{\sum_{\{u,v\} \in E(G)}w_{uv} [f(u)-f(v)]^2 + \sum_u W_u f^2(u)}{\sum_u f^2(u)}
\end{equation}
where $T^{(D)}_{i}$ is the subspace spanned by the functions $f_j$ achieving $\lambda^{(D)}_j$ for $0 \leq j \leq i$. This equation actually defines the \textit{Dirichlet eigenvalues} of the graph $G$ embedded in an appropriate host graph. In the following subsection, I will make this mapping precise.

\subsubsection{Dirichlet eigenvalues}\label{sec:Dirichlet}
For a given subgraph $S \subseteq G$, we can consider eigenfunctions of $S$ under boundary constraints and their corresponding eigenvalues. To proceed, we define the edge and vertex boundary sets 
\begin{enumerate}
    \item $\partial S = \left\{\{u,v\} \in E(G) \;\vert\; u \in V(S) , v \notin V(S) \right\}$ and 
    \item $\delta S = \left\{u \in V(G\setminus S) \;\vert\; \{u,v\} \in \partial S \; \text{for some} \; v \in V\right\}$.
\end{enumerate}
Any function $f:S\longrightarrow \mathbb{R}$ can be extended to a function $f:S\cup \delta S \longrightarrow \mathbb{R}$ with the Dirichlet boundary condition $f(u \in \delta S) = 0$ or $\restr{f}{\delta S} = 0$. \textit{Dirichlet eigenvalues} are the eigenvalues of $S$ under this boundary constraint. To be precise,
\begin{equation}\label{eqn:Dirichlet_def}
    \lambda_i^{(D)} = \inf_{\tiny{ \substack{f \perp T^{(D)}_{i-1} \\\restr{f}{\delta S} = 0}}} \frac{\sum_{\{u,v\} \in E(S) \cup \partial S}w_{uv} [f(u)-f(v)]^2}{\sum_{u \in V(S)} f^2(u)}
\end{equation}
where $T^{(D)}_{i}$ is the subspace spanned by the functions $f_j$ achieving $\lambda^{(D)}_j$ for $0 \leq j \leq i$.

Now, recall that in the previous section we had a graph $G$ with weight function $w:E(G)\longrightarrow \mathbb{R}$. We embed this graph in a host graph $G'\supseteq G$ and extend the function $w:E(G)\cup \partial G \longrightarrow \mathbb{R}$ by requiring that $W_u = \sum_{v \in \delta G}w_{uv}$. That is, if the degree of vertex $u$ in $G$ is $d_u$, then the degree of vertex $u$ in $G'$ is $d_u + W_u$. (See \Cref{fig:dirichlet}.) Now, one can explicitly impose the Dirichlet constraint on \cref{eqn:Dirichlet_def} and recover \cref{eqn:rc-pert}:
\begin{align*}
    \lambda_i^{(D)} &= \inf_{\tiny{ \substack{f \perp T^{(D)}_{i-1} \\ \restr{f}{\delta G} = 0}}} \frac{\sum_{\{u,v\} \in E(G) \cup \partial G}w_{uv} [f(u)-f(v)]^2}{\sum_{u \in V(G)} f^2(u)}\\
    &= \inf_{\tiny{f \perp T_{i-1}}} \frac{\sum_{\{u,v\} \in E(G)}w_{uv} [f(u)-f(v)]^2 + \sum_{u \in V(G)} W_u f^2(u)}{\sum_{u \in V(G)} f^2(u)}.
\end{align*}
This embedding identity is often a useful way to geometrize a physical potential and both descriptions can be useful depending upon one's goals. 
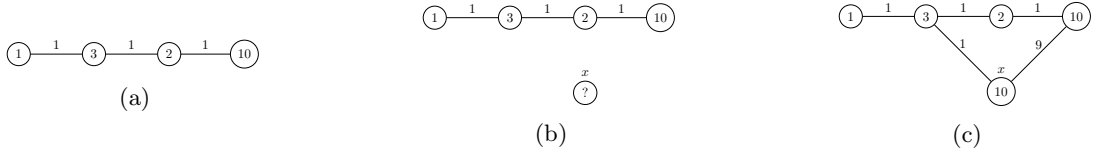
\begin{figure}[t]
    \begin{subfigure}{.33\textwidth}
    \centering
    \scalebox{.5}{%
        \begin{tikzpicture}[node distance = 2cm]
            \node[draw=black,circle] (1) {1};
            \node[draw=black,circle, right of = 1] (3) {3};
            \node[draw=black,circle, right of = 3] (2) {2};
            \node[draw=black,circle, right of = 2] (10) {10};
            
            \draw (1) -- (3) node[midway, above ]{1};
            \draw (3) -- (2) node[midway, above ]{1};
            \draw (2) -- (10) node[midway, above ]{1};
        \end{tikzpicture} %
}
    \caption{\label{fig:dirichlet1}}
    \end{subfigure}
    \begin{subfigure}{.33\textwidth}
    \centering
    \scalebox{.5}{%
        \begin{tikzpicture}[node distance = 2cm]
            \node[draw=black,circle] (1) {1};
            \node[draw=black,circle, right of = 1] (3) {3};
            \node[draw=black,circle, right of = 3] (2) {2};
            \node[draw=black,circle, right of = 2] (10) {10};
            \node[draw=black,circle, below of = 2,label={$x$}] (x) {?};
            
            \draw (1) -- (3) node[midway, above ]{1};
            \draw (3) -- (2) node[midway, above ]{1};
            \draw (2) -- (10) node[midway, above ]{1};
        \end{tikzpicture} %
}
    \caption{\label{fig:dirichlet2}}
    \end{subfigure}
    \begin{subfigure}{.33\textwidth}
    \centering
    \scalebox{.5}{%
        \begin{tikzpicture}[node distance = 2cm]
            \node[draw=black,circle] (1) {1};
            \node[draw=black,circle, right of = 1] (3) {3};
            \node[draw=black,circle, right of = 3] (2) {2};
            \node[draw=black,circle, right of = 2] (10) {10};
            \node[draw=black,circle, below of = 2,label={$x$}] (x) {10};
            
            \draw (1) -- (3) node[midway, above ]{1};
            \draw (3) -- (2) node[midway, above ]{1};
            \draw (2) -- (10) node[midway, above ]{1};
            \draw (3) -- (x) node[midway, above]{1};
            \draw (10) -- (x) node[midway, above]{9};
        \end{tikzpicture} %
}
    \caption{\label{fig:dirichlet3}}
    \end{subfigure}
    \caption{In (\subref{fig:dirichlet1}) we show the encoding of a matrix with diagonal terms $(1,3,2,10)$ and label the corresponding vertices respectively. We add a vertex $x$ where the function $f(x)=0$ in (\subref{fig:dirichlet2}). Finally, in (\subref{fig:dirichlet3}), we add edges such that each vertex is now labeled by its degree. \label{fig:dirichlet}}
\end{figure}

\subsubsection{The normalized Laplacian}
Although the expressions in \Cref{sec:Dirichlet} are sufficient to completely characterize all real matrices, we can derive a more elegant bound by perturbing the normalized Laplacian rather than the combinatorial Laplacian. We let $D = \diag{(d_u)_u}$ and define the symmetric normalized Laplacian as $\mathcal{L} = D^{-1/2}LD^{-1/2}$. Explicitly, this can be written
\[
    \mathcal{L}(u,v) = \begin{cases}
        1 & u=v \\
        -\frac{w_{uv}}{\sqrt{d_u d_v}} & u \neq v.
    \end{cases}
\]
Similar to the combinatorial Laplacian, for any function $f:V\longrightarrow \mathbb{R}$, the operator $\mathcal{L}$ satisfies
\[
    \mathcal{L}f(u) = \frac{1}{\sqrt{d_u}}\sum_{v}w_{uv}\left[\frac{f(u)}{\sqrt{d_u}}-\frac{f(v)}{\sqrt{d_v}}\right]
\]
and eigenfunctions $f$ of $\mathcal{L} + W$ satisfy 
\begin{align*}
    (\lambda-W_u)f(u) &= \frac{1}{\sqrt{d_u}} \sum_v w_{uv}\left[\frac{f(u)}{\sqrt{d_u}} - \frac{f(v)}{\sqrt{d_v}} \right].
\end{align*}

Letting $\phi = f/\sqrt{d}$,
\begin{align}\label{eqn:normalized_operator}
    (\lambda-W_u)\phi(u)d_u &= \sum_v w_{uv}\left[\phi(u) -\phi(v) \right].
\end{align}

Our treatment of \cref{eqn:combinatorial_operator,eqn:normalized_operator} can be unified by considering equations of the form
\begin{equation}\label{eqn:qweighted}
    L_q \phi(u) = (\lambda-W_u)q_u \phi(u) = \sum_v w_{uv}\left[\phi(u) -\phi(v) \right]
\end{equation}
where taking $q_u = d_u$ reproduces \cref{eqn:normalized_operator} and $q_u = 1$ reproduces \cref{eqn:combinatorial_operator}.

Hence, eigenvalues of either Laplacian are given by their respective Rayleigh quotients,
\begin{equation}
    \lambda_i = \inf_{\tiny{f \perp qT_{i-1}}} \frac{\sum_{\{u,v\}}w_{uv} [f(u)-f(v)]^2}{\sum_u q_u f^2(u)}
\end{equation}
where $T_i$ is the subspace spanned by the functions $f_j$ achieving $\lambda_j$ for $0 \leq j \leq i$. Similarly, for either Laplacian perturbed by a diagonal matrix $W$, the eigenvalues are given by
\begin{equation}\label{eqn:rc-pert2}
    \lambda_i = \inf_{\tiny{f \perp q T_{i-1}}} \frac{\sum_{\{u,v\}}w_{uv} [f(u)-f(v)]^2 + \sum_u q_u W_u f^2(u)}{\sum_u q_u f^2(u)}.
\end{equation}
This can once again be seen as Dirichlet eigenvalues as in \cref{sec:Dirichlet}, however one must be careful as the expression arising from \cref{eqn:rc-pert2} for normalized Laplacians diverges from the correct expression for Dirichlet eigenvalues of the host graph.

\subsection{The spectral gap}
Now that we have a characterization of the Dirichlet eigenvalues, we are prepared to handle the spectral gap of the operator $L_q+W$. Suppose that $\lambda_0$ has eigenfunction $\phi \geq 0$. Then, we can characterize the spectral gap of $L_q + W$ as follows.
\begin{prop}\label{prop:gap}
\begin{equation*}
    \gamma = \inf_{\tiny{g \perp q\phi^2}} \frac{\sum_{\{u,v\}}w_{uv}\phi(u)\phi(v)[g(u)-g(v)]^2}{\sum_u q_u g^2(u)\phi^2(u)}
\end{equation*}
\end{prop}
\begin{proof}
Before proceeding, we need the standard fact that for any $g:V\longrightarrow \mathbb{R}$,
\begin{equation}\label{eqn:fact1}
    \sum_{\{u,v\} } w_{uv}\left[g(u)\phi(u)-g(v)\phi(v)\right]^2 = \sum_{u} (\lambda_0-W_u) q_u g^2(u)\phi^2(u) + \sum_{\{u,v\} }w_{uv}\left[g(u)-g(v)\right]^2\phi(u)\phi(v).
\end{equation}
To see this, begin with \cref{eqn:combinatorial_operator} and write
\begin{dgroup*}
    \[
        \sum_{u} (\lambda_0-W_u) q_u g^2(u)\phi^2(u) = \sum_{u} g^2(u)\sum_{v} w_{uv}\phi(u)[\phi(u)-\phi(v)]  
    \]
    \[ = \sum_{u} \left[d_u g^2(u)\phi^2(u) - \sum_{v} w_{uv} g^2(u)\phi(u)\phi(v)\right]
    \]
    \[= \sum_{\{u,v\}}w_{uv}\left(g^2(u)\phi^2(u)+g^2(v)\phi^2(v) - \left[g^2(u)+g^2(v)\right]\phi(u)\phi(v)\right)
    \]
    \[
    = \sum_{\{u,v\}}w_{uv}\left(\left[g(u)\phi(u)-g(v)\phi(v)\right]^2 - \left[g^2(u)+g^2(v)-2g(u)g(v)\right]\phi(u)\phi(v)\right)
    \]
    \[
    =\sum_{\{u,v\}}w_{uv}\left(\left[g(u)\phi(u)-g(v)\phi(v)\right]^2 - \left[g(u)-g(v)\right]^2\phi(u)\phi(v)\right).\] 
\end{dgroup*}

With this in hand, we turn to $\lambda_1$.
\begin{align*}
    \lambda_1 &= \inf_{\tiny{f \perp q\phi}} \frac{\sum_{\{u,v\}}w_{uv} [f(u)-f(v)]^2 + \sum_{u} q_u W_u f^2(u)}{\sum_{u } q_u f^2(u)} \\
    &= \inf_{\tiny{g \perp q\phi^2}} \frac{\sum_{\{u,v\} }w_{uv} [g(u)\phi(u)-g(v)\phi(v)]^2 + \sum_{u } q_u W_u g^2(u)\phi^2(u)}{\sum_{u } q_u g^2(u)\phi^2(u)}\\
    &= \inf_{\tiny{g \perp q\phi^2}} \frac{\sum_{\{u,v\} }w_{uv}\phi(u)\phi(v)[g(u)-g(v)]^2  + \lambda_0 \sum_{u}q_u \phi^2(u) g^2(u)}{\sum_{u} q_u g^2(u)\phi^2(u)}\\
    &= \inf_{\tiny{g \perp q\phi^2}} \frac{\sum_{\{u,v\}}w_{uv}\phi(u)\phi(v)[g(u)-g(v)]^2}{\sum_{u } q_u g^2(u)\phi^2(u)} + \lambda_0.
\end{align*}

Thus, we have that
\begin{align*}
    \gamma = \inf_{\tiny{g \perp q\phi^2}} \frac{\sum_{\{u,v\}}w_{uv}\phi(u)\phi(v)[g(u)-g(v)]^2}{\sum_u q_u g^2(u)\phi^2(u)}.
\end{align*}
\end{proof}

\section{Warm-up: Cheeger upper bounds}\label{sec:cheeger}
\subsection{The Cheeger constant}
The Cheeger constant of a graph describes the graph's isoperimetric ratio, or the surface area to volume ratio of any subgraph. Noting that \Cref{prop:gap} gives an expression for the gap that is equivalent to the Rayleigh quotient of a weighted graph with weights $\omega_{uv} = w_{uv}\phi(u)\phi(v)$, we use $\omega$ as a modified weight function for defining both area and volume. That is, for a subgraph $S\subseteq G$ we let
\begin{enumerate}
    \item $\overline S = G\setminus S$,
    \item the boundary vertices $\delta S = \{u \in \overline S \;\vert\; u\sim v \in S\},$ 
    \item the surface area $\abs{\partial S} = \sum_{u \in V(S), v \in \delta S} w_{uv}\phi(u)\phi(v)$, and
    \item the volume $\vol(S) = \sum_{u \in S} q_u \phi^2(u)$.
\end{enumerate}
Then, we reproduce the weighted Cheeger constant of \cite{Chung2000} 
\begin{equation}\label{eqn:Cheeger_constant}
    h = \min_{S \subset G}\frac{\abs{\partial S}}{\min_{S' \in \{S,\overline S\}}\vol(S')}.
\end{equation}
Note that in the event that both $w_{uv}=1$ for all $\{u,v\} \in E(G)$ and $\phi \neq 0$ is any trivial function, this reproduces the ratio
\[
    \frac{\text{\# edges in $\partial S$}}{\text{\# vertices in S}}
\]
which is the standard Cheeger constant for an unweighted graph. 

\subsection{The upper bound}
\Cref{prop:gap} instructs us that we can use any function $g\perp q\phi^2$ to upper bound the gap and \Cref{prop:Hermitian} allows us to ignore the case that $H$ is not real. Thus, the upper bound derives from simply choosing an appropriate trial function in \Cref{prop:gap}. 
\begin{thm}\label{thm:upper}
    For any $H=L+W$ with ground-state $\phi$ corresponding to weighted Cheeger constant $h$
    \[
        \gamma \leq 2h.
    \]
\end{thm}
\begin{proof}
For $S$ achieving the infimum in \cref{eqn:Cheeger_constant}, we put the function 
\[
    g(u) = \begin{cases}
        \vol(\overline S) & u \in S\\
        -\vol(S) & u \notin S.
    \end{cases}
\]
into \cref{prop:gap}. Without loss of generality, we assume that $\vol(S)\leq \vol(\overline{S})$ and find that
\begin{align*}
    \gamma &\leq \frac{\sum_{\{u,v\} \in E(G)}w_{uv}\phi(u)\phi(v)[g(u)-g(v)]^2}{\sum_u q_u g^2(u)\phi^2(u)}\\
    &=\frac{\left(\sum_{\{u,v\} \in \partial S}w_{uv}\phi(u)\phi(v)\right)[\vol(S) + \vol(\overline S)]^2}{\vol(\overline S)^2\sum_{u \in V(S)} q_u \phi^2(u) + \vol(S)^2\sum_{u \in V(\overline S)} q_u \phi^2(u)}\\
    &\leq \frac{(h \vol(S)) [\vol(S) + \vol(\overline S)]^2}{\vol(S)[\vol(S)^2 + \vol(\overline S)^2]}\\
    &= h \frac{[\vol(S) + \vol(\overline S)]^2}{\vol(S)^2 + \vol(\overline S)^2}\\
    &\leq 2 h.
\end{align*}
\end{proof}
\Cref{thm:upper} also holds for all Hermitian matrices by \cref{prop:Hermitian}.

\section{Removing negative edge weights}\label{sec:neg_edges}
In this section, I provide a theorem relating the spectrum of the graph $G=(V,E)$ with edge weights $w:E \longrightarrow \mathbb{R}$ to the graph $G^+ = (V,E\setminus E^-)$. For edges with negative edge weights and endpoints $(x,y)$, we consider the set of paths from $x$ to $y$ through $G^+$, denoted $P(x,y)$. Thus, a path from $x$ to $y$ is a member of the set $P(x,y)$. The strategy behind this theorem is to consider an edge $\{u,v\}$ with weight $w_{uv}<0$ as in \cref{fig:graph1}. Then, we find some path connecting $u$ and $v$ that traverses $G^+$ and route the negative weights along this path. Routing is not an uncommon approach (see, e.g. \cite{diaconis1991geometric}) and has a lot in common with the method of proving Poincar\'{e} inequalities \cite{Chung}.

\begin{figure}[t]
    \begin{subfigure}{.33\textwidth}
    \centering
    \begin{tikzpicture}
        \foreach \x in {0,...,4}{
            \foreach \y in {0,...,4}{
                \draw[line width=.5pt] (\y,0) -- (\y,4);
                \draw[line width=.5pt] (0,\x) -- (4,\x);
            }
        }
        \draw[line width=1.5pt,color=red,dashed] (2,2) -- (2,3) node[midway,right] {\scalebox{.6}{$-\frac{1}{3}$}};
        \foreach \x in {0,...,4}{
            \foreach \y in {0,...,4}{
                \node[circle,fill,scale=0.5] at (\x,\y) {};
            }
        }
        \node[dot={200}{x}] at (2,2) {};
        \node[dot={100}{y}] at (2,3) {};
    \end{tikzpicture}
     \caption{\label{fig:graph1}}
    \end{subfigure}
    \begin{subfigure}{.33\textwidth}
    \centering
    \begin{tikzpicture}
        \foreach \x in {0,...,4}{
            \foreach \y in {0,...,4}{
                \draw[line width=.5pt] (\y,0) -- (\y,4);
                \draw[line width=.5pt] (0,\x) -- (4,\x);
            }
        }
        \draw[line width=1.5pt,color=red,dashed] (2,2) -- (2,3) node[midway,right] {\scalebox{.6}{$-\frac{1}{3}$}};
        \draw[line width=1.5pt,color=blue,dashed] (2,2) -- (3,2);
        \draw[line width=1.5pt,color=blue,dashed] (3,2) -- (3,3);
        \draw[line width=1.5pt,color=blue,dashed] (3,3) -- (2,3);
        \foreach \x in {0,...,4}{
            \foreach \y in {0,...,4}{
                \node[circle,fill,scale=0.5] at (\x,\y) {};
            }
        }
                \node[dot={200}{x}] at (2,2) {};
        \node[dot={100}{y}] at (2,3) {};
    \end{tikzpicture}
     \caption{\label{fig:graph2}}
    \end{subfigure}
    \begin{subfigure}{.33\textwidth}
    \centering
    \begin{tikzpicture}
        \foreach \x in {0,...,4}{
            \foreach \y in {0,...,4}{
                \draw[line width=.5pt] (\y,0) -- (\y,4);
                \draw[line width=.5pt] (0,\x) -- (4,\x);
            }
        }
        \draw[line width=1.5pt,color=white] (2,2) -- (2,3);
        \draw[line width=1.5pt,color=white] (2,2) -- (3,2);
        \draw[line width=1.5pt,color=white] (3,2) -- (3,3);
        \draw[line width=1.5pt,color=white] (3,3) -- (2,3);
        \foreach \x in {0,...,4}{
            \foreach \y in {0,...,4}{
                \node[circle,fill,scale=0.5] at (\x,\y) {};
            }
        }
                \node[dot={200}{x}] at (2,2) {};
        \node[dot={100}{y}] at (2,3) {};
    \end{tikzpicture}     \caption{\label{fig:graph3}}
    \end{subfigure}
    \caption{A graph $G$ such that $w_{uv}=1$ for all edges $\{u,v\}$ except for $\{x,y\}$. In (\subref{fig:graph1}), the negative edge weight $w_{xy}=-1/3$ is identified. In (\subref{fig:graph2}), we identify some path from $P(x,y)$ that does not cross the negative weighted edge $\{x,y\}$. In (\subref{fig:graph3}) we redistribute the negative edge weight along the path, inducing a new graph that lower bounds the gap of $G$. \label{fig:rerouting}}
\end{figure}
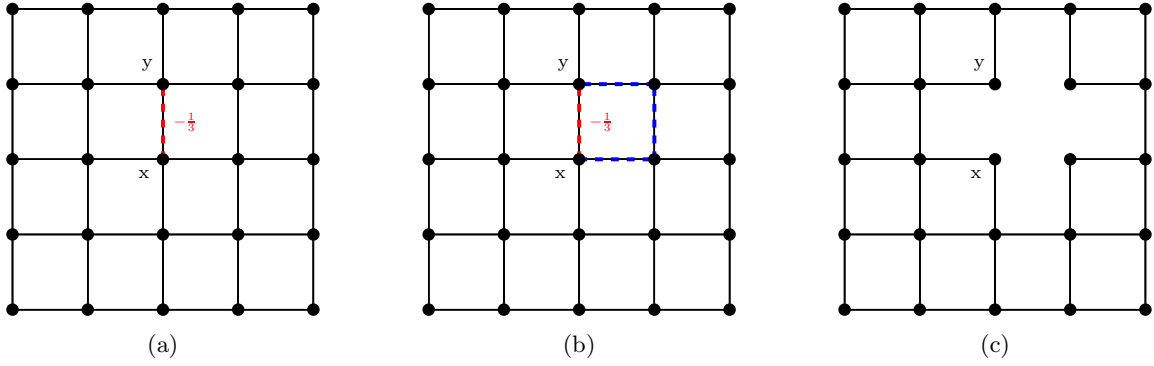

\begin{thm}\label{thm:positive}
    Suppose that for a graph $S \subseteq G$, $S^+$ is connected and there exists an $\displaystyle \alpha : \bigcup_{(x,y)\in E(S)}P(x,y) \longrightarrow [0,1]$ such that for any $\{u,v\} \in E(S)$,
    \begin{enumerate}
        \item $\displaystyle \sum_{\tiny{p \in P(u,v)}} \alpha_{p} = 1$;
        \item and $0 < \omega_{uv} = \displaystyle w_{uv} - \sum_{\tiny{w_{xy}<0}}\sum_{\tiny{\substack{p \in P(x,y) \\ (u,v) \in p}}} \abs{w_{xy}}\ell_p \alpha_p$,
    \end{enumerate}
     then, for each $i$, there exists an $\widetilde\omega \geq \omega$ such that
    \begin{align*}
        \lambda_i^{(D)} &= \inf_{\tiny{ \substack{f \perp qT^{(D)}_{i-1} \\ \restr{f}{\delta S} = 0}}} \frac{\sum_{\{u,v\} \in E(S^+)\cup \partial S}\widetilde\omega_{uv} [f(u)-f(v)]^2}{\sum_{u \in V(S)} q_u f^2(u)}.
    \end{align*}
    and $\lambda^{(D)}_0$ is unique.
\end{thm}
\begin{proof}
Consider
\begin{align*}
    \lambda_i^{(D)} &= \inf_{\tiny{ \substack{f \perp qT^{(D)}_{i-1} \\ \restr{f}{\delta S} = 0}}} \frac{\sum_{\{u,v\} \in E(S)\cup \partial S}w_{uv} [f(u)-f(v)]^2}{\sum_{u \in V(S)} q_u f^2(u)}\\
    &\geq \inf_{\tiny{ \substack{f \perp qT^{(D)}_{i-1} \\ \restr{f}{\delta S} = 0}}} \frac{\sum_{\{u,v\} \in E(S^+)\cup \partial S}\left(w_{uv}-\sum_{\tiny{w_{xy}<0}}\sum_{\tiny{\substack{p \in P(x,y) \\ (u,v) \in p}}} \abs{w_{xy}}\ell_p \alpha_p\right) [f(u)-f(v)]^2}{\sum_{u \in V(S)} q_u f^2(u)}\\
    &=\inf_{\tiny{ \substack{f \perp qT^{(D)}_{i-1} \\ \restr{f}{\delta S} = 0}}} \frac{\sum_{\{u,v\} \in E(S^+)\cup \partial S}\omega_{uv} [f(u)-f(v)]^2}{\sum_{u \in V(S)} q_u f^2(u)}
\end{align*}
where we have applied Jensen's inequality. Thus, there exists some $ \widetilde\omega \geq \omega$ such that
\begin{align*}
    \lambda_i^{(D)} &= \inf_{\tiny{ \substack{f \perp qT^{(D)}_{i-1} \\ \restr{f}{\delta S} = 0}}} \frac{\sum_{\{u,v\} \in E(S^+)\cup \partial S}\widetilde\omega_{uv} [f(u)-f(v)]^2}{\sum_u q_u f^2(u)}.
\end{align*}
\end{proof}

Furthermore, since this is just the Rayleigh quotient corresponding to the Dirichlet eigenvalues of a connected graph, the Perron-Frobenius theorem applies and we also have that $\lambda_0^{(D)}$ is unique. \Cref{thm:positive} also applies to the characterization of $\gamma$ in \cref{prop:gap}:
\begin{cor}\label{cor:positive}
    Suppose that for a graph $G=(V,E)$, $G^+$ is connected and there exists an $\displaystyle \alpha : \bigcup_{(x,y)\in E}P(x,y) \longrightarrow [0,1]$ such that for any $(u,v) \in E$
    \begin{enumerate}
    \item $\displaystyle \sum_{\tiny{p \in P(u,v)}} \alpha_{p} = 1$ and
    \item $\displaystyle \widetilde\omega_{uv} > w_{uv}\phi(u)\phi(v) - \sum_{\tiny{w_{xy}<0}}\sum_{\tiny{\substack{p \in P(x,y) \\ \{u,v\} \in p}}} \abs{w_{xy}}\phi(x)\phi(y)\ell_p \alpha_p$,
\end{enumerate}
where $\ell_{p} = \abs{p}$ is the length of path $p$,
\begin{equation}\label{eqn:rgap}
    \gamma = \inf_{\tiny{g \perp q\phi^2}} \frac{\sum_{\{u,v\} \in E(G^+)}\widetilde\omega_{uv}[g(u)-g(v)]^2}{\sum_u q_u g^2(u)\phi^2(u)}.
\end{equation}
\end{cor}

The unsightliness of $\widetilde\omega$ is not lost on me. Nonetheless, the expression is quite intuitive. Basically, a potentially useful Cheeger-type bound exists whenever we can redistribute negative weighted edges along paths through $G^+$ connecting them. I present this form, however, because it is unlikely that in practical situations we will be faced with something that can be easily routed along a single path. Such a statement is easy to derive by choosing unique paths satisfying the constraints of \cref{thm:positive}, however. \cref{cor:simpler} provides one such simplification.

\begin{remark}
Although there exist cases where one can create cuts such that condition 2 above is always unachievable, in many cases, this is handled by the unitary rotation considered in \cref{sec:rayleigh}.
\end{remark}

\section{Two Dirichlet Cheeger inequalities}
In this section, I present Cheeger inequalities using a technique similar to \cite{Chung2000}. Unlike \cite{Chung2000}, we wish to construct an inequality for as broad a class of matrices as possible. A theorem similar to Theorem 1 was originally derived and presented by me in \cite{Jarret2014}, however, at that time, I did not realize that it could be significantly strengthened to the more useful one below. First, we need to bound the contribution of the term $W$ to the eigenvalues $\lambda_i$. Because of \cref{thm:positive,cor:positive}, we only need to consider the case of nonnegative edge weights.

\begin{lemma}\label{lem:potential}
    For a graph $G=(V,E)$, suppose $\phi : V \longrightarrow \mathbb{R}$ satisfies
    \begin{equation}\label{eqn:potential_a1}
            \left(\lambda -W_u\right)q_u \phi(u) = \sum_{v\sim u} w_{uv} \left[\phi(u)-\phi(v) \right].
    \end{equation}
    for $w > 0$. Then,  
    \begin{align*}
        \lambda &\geq \max_{S' \in \{S,V\setminus S\} }\left( \frac{\sum_{u \in S'} \left(W_u +\lambda_0^D(S')\right) q_u \phi(u)^2}{\sum_{u \in S'} q_u\phi(u)^2}\right)
    \end{align*}
    for $S = \left\{u \in V \;\vert\; \phi(u) \geq 0 \right\}$ and $\lambda_0^D(S')$ the lowest Dirichlet eigenvalue of $S' \subseteq G$.
\end{lemma}

\begin{proof}
Without loss of generality, assume that $S' \in \{S, V \setminus S\}$ achieves the maximum above. Now,
    \begin{dgroup*}
    \[
        \sum_{u \in S'} (\lambda - W_u)q_u \phi(u)^2 = \sum_{u \in S'}\sum_{v \sim u} w_{uv} (\phi(u)-\phi(v))\phi(u)
    \]
    \[= \sum_{\{u,v\} \in E(S')}w_{uv} (\phi(u)-\phi(v))^2 + \sum_{\substack{\{u,v\} \in \partial S' \\ u \in S'}}w_{uv}(\phi(u)-\phi(v))\phi(u)
    \]
    \[\geq \lambda_0^{D}(S')\sum_{u \in S'} q_u\phi^2(u) - \sum_{\substack{\{u,v\} \in \partial S' \\ u \in S'}}w_{uv}\phi(v)\phi(u)
    \]
    \[
        \geq \lambda_0^{D}(S')\sum_{u \in S'} q_u\phi^2(u).
    \]
    \end{dgroup*}
    Above, the first inequality follows from the definition of the Dirichlet eigenvalues and the second because ${\phi(S') \phi(\overline{S'}) \leq 0}$.
\end{proof}

\begin{cor}\label{cor:potential}
    For a graph $G=(V,E)$, suppose $\phi : V \longrightarrow \mathbb{R}$ satisfies
    \begin{equation}
            \left(\lambda -W_u\right)q_u \phi(u) = \sum_{v\sim u} w_{uv} \left[\phi(u)-\phi(v) \right].
    \end{equation}
    for $w > 0$. Then,  
    \begin{align*}
        \lambda &\geq \max_{S' \in \{S,V\setminus S\} }\left( \frac{\sum_{u \in S'} W_u q_u \phi(u)^2}{\sum_{u \in S'} q_u\phi(u)^2}\right)
    \end{align*}
    for $S = \left\{u \in V \;\vert\; \phi(u) \geq 0 \right\}$.
\end{cor}

\Cref{cor:potential} allows us to derive our primary Cheeger inequality, which generalizes from \cite{Chung2000}:

\begin{thm}\label{thm:cheeger}
    Suppose $\phi_i : V \longrightarrow \mathbb{R}$, satisfy
    \begin{equation}\label{eqn:assumption}
            \left[\lambda_i -W_u\right]q_u \phi_i(u) = \sum_{v \sim u} w_{uv} \left[\phi_i(u)-\phi_i(v) \right]
    \end{equation}
    and let $\gamma = \lambda_1 - \lambda_0$. Then,
    \[
        \gamma \geq \sqrt{h^2 + Q^2} - Q
    \]
    where
    \[
        Q = \frac{\sum_{u \in S} d_u\phi_1^2(u)}{\sum_{u \in S}q_u\phi_1^2(u)}.
    \]
\end{thm}
\begin{proof}
    For a particular vertex $u_0$, we begin by considering the one-parameter family
    \[
        f_\epsilon(u) = \begin{cases}
            f(u_0) + \epsilon \vol\left(G\setminus\{u_0\}\right) & u = u_0 \\
            f(u) - \epsilon q_{u_0} \phi_0^2(u_0) & \text{otherwise}
        \end{cases}
    \]
    where $f$ achieves the infimum in \cref{eqn:rgap}. Clearly, $f_\epsilon$ satisfies $f_\epsilon \perp q\phi_0^2$. Then, we introduce this into the Rayleigh quotient $R(f_\epsilon)$ and note that $\frac{d}{d\epsilon}R(f_\epsilon) \vert_{\epsilon = 0} = 0$ \footnote{In an earlier version of this paper, I did not pursue a variational approach, joking that I was not a masochist. However, masochism seems inevitable, as the previous approach was inconsistent with \cref{eqn:rgap}.} 
    \[
      0 = \restr{\frac{d R(f(\epsilon))}{d \epsilon}}{\epsilon = 0} = \frac{d}{d\epsilon} \left[ \frac{\displaystyle\sum_{\{u,v\}}w_{uv}\phi_0(u)\phi_0(v)[f_\epsilon(u)-f_\epsilon(v)]^2}{\displaystyle\sum_u q_u f_\epsilon^2(u)\phi_0^2(u)} \right]_{\epsilon = 0}
        = \frac{d}{d\epsilon} \left[ \frac{\displaystyle\sum_{\substack{\{u,v\} \\ u,v \neq u_0}}\omega_{uv}[f(u)-f(v)]^2 + \sum_{u\neq u_0}\omega_{u_0 u}\left(f(u_0)-f(u) + \epsilon \vol(G) \right]_{\epsilon = 0}^2 }{\displaystyle\sum_{u\neq u_0} q_u \left(f(u)-\epsilon q_{u_0}\phi_0^2(u_0) \right)^2\phi_0^2(u) + q_{u_0}\left(f(u_0)+ \epsilon \vol(G\setminus \{u_0\}) \right)^2} \right]_{\epsilon = 0}
        = \frac{2\sum_{u\neq u_0}w_{u_0 u}\left(f(u_0)-f(u) \right)\vol(G)}{\sum_u q_u f^2(u)\phi_0^2(u)} - 2R(f)q_{u_0}\phi_0^2(u_0) \left( \frac{-\sum_{u\neq u_0} q_u f(u)\phi_0^2(u)  + f(u_0)\vol\left(G\setminus \{u_0\}\right)}{{\sum_u q_u f^2(u)\phi_0^2(u)}}\right)
        =\sum_{u\neq u_0}w_{u_0 u}\left(f(u_0)-f(u) \right)\vol(G) - \gamma q_{u_0}\phi_0^2(u_0) \left({f(u_0)\vol(G\setminus \{u_0\}) - \sum_{u\neq u_0} q_u f(u) \phi_0^2(u)}\right)
        =\sum_{u\neq u_0}w_{u_0 u}\left(f(u_0)-f(u) \right)\vol(G) - \gamma q_{u_0}\phi_0^2(u_0) \left( f(u_0)\vol(G\setminus \{u_0\}) + q_{u_0}f(u_0)\phi_0^2(u_0)  \right)
      =\sum_{u\neq u_0}w_{u_0 u}\left(f(u_0)-f(u) \right)\vol(G) - \gamma q_{u_0}f(u_0)\phi_0^2(u_0)\vol(G)
      =\sum_{u\neq u_0}w_{u_0 u}\left(f(u_0)-f(u) \right) - \gamma q_{u_0}f(u_0)\phi_0^2(u_0) .
    \]
    Thus, for any $u$, $f(u)$ satisfies
    \begin{align*}
        \gamma q_u f(u) \phi_0^2(u) &= \sum_{v \sim u} w_{uv}\phi_0(v)\phi_0(u)[f(u)-f(v)]\\
        \gamma q_u f^2(u) \phi_0^2(u) &= \sum_{v \sim u} w_{uv}\phi_0(v)\phi_0(u)[f(u)-f(v)]f(u).
    \end{align*}
    Let $S \subseteq G$ be the subgraph of $G$ induced by the vertex set $V(S) = \left\{v \vert \phi_1(v) \geq 0 \right\}$ and let $\omega_{uv} = w_{uv}\phi_0(u)\phi_0(v)$. Without loss of generality, we assume that $\sum_{u \in S}q_u \phi_0^2(u) \leq \sum_{u \notin S} q_u \phi_0^2(u)$. (If this is not the case, simply take $f \mapsto -f$.) Then, for any region $S' \subseteq G$ such that either $S' \subseteq S$ or $\overline{S'} \subseteq S$, and define the Cheeger ratio as
    \begin{align*}
        h_{S'} &\equiv \frac{\abs{\partial S'}}{\min\{\vol(S'),\vol(\overline{S'}\}} \\
        &= \begin{cases}\frac{\abs{\partial S'}}{\sum_{u \in V(S')} q_u \phi_0^2(u)} & S' \subseteq S\\
        \frac{\abs{\partial S'}}{\sum_{u \in V(\overline{S'})} q_u \phi_0^2(u)} & \text{$\overline{S'}\subseteq S$}
        \end{cases}\\
        &\geq h \numberthis\label{eqn:local_Cheeger}.
    \end{align*}
    Now, we let
        \begin{align*}
            \gamma\sum_{u\in V(S)}q_uf^2(u)\phi_0^2(u) &= \sum_{u \in V(S)} \sum_{v \sim u}\omega_{uv}[f(u)-f(v)]f(u) \\
            &= \sum_{\{v,u\} \in E(S)}\omega_{uv}[f(u)-f(v)]^2 + \sum_{\substack{\{u,v\} \in \partial S \\ u \in V(S)}}\omega_{uv}[f(u)-f(v)]f(u) \\
            &\geq \sum_{\{v,u\} \in E(S)}\omega_{uv}[f(u)-f(v)]^2 + \sum_{\substack{\{u,v\} \in \partial S \\ u \in V(S)}} \omega_{uv}f^2(u)
        \end{align*}
    since $f(u)f(v)\leq0$ whenever $\{u,v\} \in \partial S$.
    
    Introducing the function
    \[
        g(u) = \begin{cases}
            f(u) & f(u) \geq 0 \\
            0 & \text{otherwise},
        \end{cases}
    \]
    we have that
    \begin{dgroup*}
    \[
        {\gamma \geq \Phi} = \frac{\displaystyle\sum_{\{v,u\}}\omega_{uv}[g(u)-g(v)]^2}{\displaystyle\sum_{u}q_u g^2(u)\phi_0^2(u)}
    \]
    \[    %
        =\frac{\displaystyle\sum_{\{v,u\}}\omega_{uv}[g(u)-g(v)]^2}{\displaystyle\sum_{u\in V(S)}q_uf^2(u)\phi_0^2(u)} \cdot \frac{\displaystyle\sum_{\{v,u\}}\omega_{uv}[g(u)+g(v)]^2}{\displaystyle\sum_{\{v,u\}}\omega_{uv}[g(u)+g(v)]^2}\\
    \]    %
    \[    \geq     \frac{\left(\displaystyle\sum_{\{v,u\}} \omega_{uv}\abs{g^2(u)-g^2(v)}\right)^2}{\left(\displaystyle\sum_{u\in V(S)}q_uf^2(u)\phi_0^2(u)\right) \left(\displaystyle\sum_{\{v,u\}}\omega_{uv}[g(u)+g(v)]^2\right)}\\
    \]
    \[%
        = \frac{\left(\displaystyle\sum_{\{v,u\}} \omega_{uv}\abs{g^2(u)-g^2(v)}\right)^2}{\left(\displaystyle\sum_{u\in V(S)}q_uf^2(u)\phi_0^2(u) \right)\left( \displaystyle 2 \sum_{u \in V(S)}f^2(u)\phi_0(u)\sum_{v \sim u}w_{uv} \phi_0(v) -\sum_{\{v,u\}}\omega_{uv}[g(u)-g(v)]^2\right)}
    \]    %
    \[= \frac{\left(\displaystyle\sum_{\{v,u\}}\omega_{uv} \abs{g^2(u)-g^2(v)}\right)^2}{\left(\displaystyle\sum_{u\in V(S)}q_u g^2(u)\phi_0^2(u) \right)\left( \displaystyle 2 \sum_{u \in V(S)}f^2(u)\phi_0^2(u)q_u\left(W_u+\frac{d_u}{q_u} -\lambda_0 \right) -\sum_{\{v,u\}}\omega_{uv}[g(u)-g(v)]^2\right)}
    \]
    \[= \frac{\left(\displaystyle\sum_{\{v,u\}} \omega_{uv}\abs{g^2(u)-g^2(v)}\right)^2}{\left(\displaystyle\sum_{u \in V(S)}q_u g^2(u)\phi_0^2(u) \right)^2 \left( \frac{\displaystyle 2 \sum_{u \in V(S)}q_u\phi_1^2(u)\left(W_u+\frac{d_u}{q_u} -\lambda_0 \right)}{\displaystyle\sum_{u\in V(S)}q_u\phi_1^2(u)} - \Phi \right)}\\
    \]     %
    \begin{dmath} \label{eqn:final}    \geq \frac{\left(\displaystyle\sum_{\{v,u\}}\omega_{uv} \abs{g^2(u)-g^2(v)}\right)^2}{\left(\displaystyle\sum_{u\in V(S)}q_uf^2(u)\phi_0^2(u) \right)^2 \left( 2 \gamma + 2 Q - \Phi \right)}
    \end{dmath}
    \end{dgroup*}
    where the first inequality follows from Cauchy-Schwarz and the final inequality follows from \cref{cor:potential}. Now, suppose that we label our vertices $u_i$ with integers $i\geq 1$ such that $f(u_{i+1}) \geq f(u_i)$. Then, clearly, for any $j < i$
    \begin{align*}
        g(u_i)-g(u_j) &= \sum_{k=j}^{i-1}(g(u_{k+1}) - g(u_k)).
    \end{align*}
    Now, consider the cut $S_k = \left\{u_j \;\vert\; j \leq k \right\}$,
    \begin{align*}
        \omega_{u_i u_j} \abs*{g^2(u_i)-g^2(u_j)} &= \omega_{u_i u_j}\sum_{k=j}^{i-1} \abs*{g^2(u_{k+1})-g^2(u_k)}\\
        \sum_{j < i}\omega_{u_i u_j} \abs*{g^2(u_i)-g^2(u_j)} &= \sum_{j < i}\sum_{k=j}^{i-1} \omega_{u_i u_j} \abs*{g^2(u_{k+1})-g^2(u_k)}\\
        &=\sum_{k \leq \abs*{V}-1} \abs*{g^2(u_{k+1}) -g^2(u_{k}) }\sum_{j \leq k < i} \omega_{u_i u_j}\\
        &\geq \sum_{k \leq \abs*{V}-1} \abs*{g^2(u_{k+1}) -g^2(u_{k}) }\left(h_{S_k}\sum_{j > k} \phi_0^2(u_j)q_{u_j} \right)\\ 
        &\geq h\sum_{k \leq \abs*{V}} q_{u_k} g^2(u_k) \phi_0^2(u_k) \\&=h\sum_{u\in V(S)} q_u f^2(u) \phi_0^2(u).
    \end{align*}
    Above, both inequalities follow from \cref{eqn:local_Cheeger}, where the second also utilizes summation by parts. Introducing this into \cref{eqn:final},
    \begin{dgroup*}
    \[
        \Phi \geq  \frac{\left(\displaystyle\sum_{\{v,u\}}\omega_{uv}
        \abs*{g^2(u)-g^2(v)}\right)^2}{\left(\displaystyle\sum_{u\in V(S)}q_uf^2(u)\phi_0^2(u) \right)^2 \left( 2 \gamma + 2 Q - \Phi \right)}
    \]
    \[    \geq h^2 \frac{\left(\sum_{u \in V(S)} q_u f^2(u_u) \phi_0^2(u_u) \right)^2}{\left(\sum_{u\in V(S)}q_u f^2(u)\phi_0^2(u) \right)^2 \left( 2 \gamma + 2 Q - \Phi \right)}
    \]
    \[
        =\frac{h^2}{2 \gamma + 2Q-\Phi}.
    \]
    \end{dgroup*}

    Now,  
    \begin{dgroup*}
    \[
        h^2 \leq (2 \gamma + 2 Q -\Phi)\Phi
    \]
    \[
        =2(\gamma +Q)\Phi -\Phi^2
    \]
    \[
        \leq 2 Q \gamma - (\Phi-\gamma)^2 + \gamma^2
    \]
    \[
        \leq 2 Q \gamma + \gamma^2,
    \]
    \end{dgroup*}
    so that $\gamma \geq \sqrt{h^2+Q^2}-Q$.
\end{proof}

\begin{remark}At this point, it is worth pausing to recognize just how much tighter the bound one finds from \cref{thm:cheeger} is than its expansion around $h=0$. Had we simply assumed $h$ was small, we would have arrived at the inequality $\gamma \geq h^2/2 Q - h^4/8Q^3$. For $W=0$, one would expect the inequality $\gamma \geq h^2/(2Q)$, so our result is only slightly weaker than anticipated. At first glance, one might expect this to be our desired bound. Unlike the $W=0$ case, however, we \textit{do not} expect that $h$ will usually be small. In fact, for strongly peaked distributions, we expect that $h$ can be rather large. Thus, retaining the expression of \cref{thm:cheeger} can be essential to using this bound for most choices of $W$.
\end{remark}

The following inequality looks more like the standard Cheeger inequality and does not turn negative, however it is weak for large $h$. It follows immediately from the inequality $2(\sqrt{x+1}-\sqrt{x}) > 1/\sqrt{x+1}$ when $x>0$.
\begin{cor}\label{cor:cheeger2}
    Suppose $\phi_i : V \longrightarrow \mathbb{R}$, satisfy
    \begin{equation}
            \left[\lambda_i -W_u\right]q_u \phi_i(u) = \sum_{v \sim u} w_{uv} \left[\phi_i(u)-\phi_i(v) \right]
    \end{equation}
    and let $\gamma = \lambda_1 - \lambda_0$. Then,
    \[
        \gamma \geq \frac{h^2}{2\sqrt{h^2+Q^2}}
    \]
    where $Q$ is as in \cref{thm:cheeger}.
\end{cor}

We can now adapt \cref{cor:nonstoq} to provide a Cheeger inequality for the case considered in \cref{cor:positive}.
\begin{thm}\label{thm:nonstoq}
For a graph $G=(V,E)$, suppose
    \[
        \gamma = \inf_{\tiny{g \perp q\phi^2}} \frac{\sum_{\{u,v\} \in E(G^+)}\widetilde\omega_{uv}[g(u)-g(v)]^2}{\sum_u q_u g^2(u)\phi^2(u)},
    \]
    then
    \[
        \gamma \geq (Q + \rho) - \sqrt{(Q+\rho)^2 - h^2}
    \]
    where $\rho = \lambda_{\abs{V}-1}-\lambda_0$.
\end{thm}
For a proof, see \cref{ap:proof}. Now, \cref{thm:nonstoq} with the appropriate choice of $Q$ yields the following corollaries: 
\begin{cor}\label{cor:combinatorial}
    Suppose $H=L+W$ is an $n\times n$ real symmetric matrix with eigenvalues $\lambda_0 \leq \lambda_1 \leq \dots \leq \lambda_{n-1}$ and corresponding ground-state $\phi$ where $L$ is the combinatorial Laplacian of $G$. Then, if $G^+$ has degree at most $d_{\max}$
    \[
        \lambda_1-\lambda_0 \geq (d_{\max} + \rho) - \sqrt{(d_{\max}+\rho)^2 - h^2}
    \]
    for $\rho = \lambda_{N-1}-\lambda_0$ and 
    \[h = \sup_{\substack{\alpha>0 \\ \sum_{p \in P(u,v)} \alpha_p = 1}} h_\alpha,\]
    \[
        h_\alpha = \min_S \max_{S' \in \{S, \overline S\}}\frac{\sum_{\{u,v\} \in \partial S} \omega_{uv}(\alpha)}{\sum_{u \in S'} \phi^2(u)}
    \]
    where
    \[
    \omega_{uv}(\alpha) = \left(w_{uv}\phi(u)\phi(v)-\sum_{w_{xy}<0}\sum_{\tiny{\substack{p \in P(x,y)\\\{u,v\} \in p}}}\abs*{w_{xy}}\ell_p\alpha_p\phi(x)\phi(y) \right)
    \]
    as in \cref{thm:positive}.
\end{cor}

\begin{cor}\label{cor:normalized}
    Suppose $H=\mathcal{L}+W$ is a real symmetric matrix with eigenvalues $\lambda_0 < \lambda_1 \leq \dots \leq \lambda_{N-1}$ and corresponding ground-state $\phi$ where $\mathcal{L}$ is the normalized Laplacian of $G$. Then,
    \[
        \lambda_1 -\lambda_0 \geq (1 + \rho) - \sqrt{(1+\rho)^2 - h^2}
    \]
    for $\rho = \lambda_{N-1}-\lambda_0$ and the distributed Cheeger constant 
    \[h = \sup_{\substack{\alpha>0 \\ \sum_p \alpha_p = 1}} h_\alpha,\]
    \[
        h_\alpha = \min_S \max_{S' \in \{S, \overline S\}}\frac{\sum_{\{u,v\} \in \partial S} \omega_{uv}(\alpha)}{\sum_{u \in S'} d_u\phi^2(u)}
    \]
    where
    \[
    \omega_{uv}(\alpha) = \left(\omega_{uv}\phi(u)\phi(v)-\sum_{\omega_{xy}<0}\sum_{\tiny{\substack{p \in P(x,y)\\\{u,v\} \in p}}}\abs{\omega_{xy}}\ell_p\alpha_p \phi(x)\phi(y) \right)
    \]
    as in \cref{thm:positive}.
\end{cor}
The form of \cref{cor:combinatorial} and \cref{cor:normalized} is not as elegant as \cref{thm:cheeger}, but we shouldn't be turned off so easily; each corollary has a pleasing interpretation. Begin by taking negative edge weights and redistribute them along positive paths as best you can. The Cheeger constant of the resulting graph is always a lower bound for the gap.

\section{Applications}\label{sec:applications}
\subsection{Some simple reductions}
The approach of \cref{sec:neg_edges} is more general than one might desire. All we have effectively done in that section is apply Jensen's inequality. Restricting to unique paths, we have the following corollary to \cref{thm:positive}.
\begin{cor}\label{cor:simpler}
    Suppose that for the graph $G=(V,E)$, $\gamma(G)$ is the spectral gap of the combinatorial Laplacian of $G$. Then, if $G^+$ is connected and there exists a set of non-overlapping paths such that
    \begin{equation*}
        \mathcal{P} = \{P(u,v) \in E^+ \;\vert\; \{u,v\} \in E^- \text{ and } w(e \in P(u,v)) - \abs{w_{uv}}\abs{P(u,v)} \geq 0  \}.
    \end{equation*}
    Then, $\gamma(G) \geq \gamma(G\setminus \mathcal{P})$ where all constants are as in \cref{sec:cheeger}.
\end{cor}
Comparison theorems like this are rather easy to derive by choosing the appropriate set of paths through $G$. One can also use this to derive a Cheeger inequality that uses the Cheeger constant $h$ of $G^+$ on both sides. Note that, we obtain a tighter bound than that of \cref{thm:nonstoq}, since $\lambda_0(G) = \lambda_0(G^+) = 0$, so we only need to bound $\lambda_1(G) \geq \lambda_1(G\setminus \mathcal{P})$ and we can apply \cref{thm:cheeger} directly. 
\begin{thm}
    Suppose that the graph $G=(V,E)$, $\gamma(G)$ is spectral gap of the combinatorial Laplacian of $G$. Then, if $G^+$ is connected and there exists a set of non-overlapping paths such that
    \begin{equation*}
        \mathcal{P} = \{P(u,v) \in E^+ \;\vert\; \{u,v\} \in E^- \text{ and } w(e \in P(u,v)) - \abs{w_{uv}}\abs{P(u,v)} \geq \epsilon  \}.
    \end{equation*}
    Then, 
    \begin{equation*}
        2h \geq \gamma \geq \epsilon(\sqrt{h^2+Q^2}-Q)
    \end{equation*}
    where all constants are as in \cref{sec:cheeger}.
\end{thm}
\begin{proof}
    This follows readily from \cref{thm:cheeger}. First, it is obvious that the degree $Q'$ resulting from routing negative edge weights must satisfy $Q'\geq \epsilon Q$. Thus, one must only show that $k$, the weighted Cheeger constant of $G^+$ after routing, satisfies $k \geq \epsilon h$. If we let $\omega$ be the edge-weights after appropriately routing the original weight function $w$,
    \begin{equation*}
        k = \frac{\sum_{\substack{u \in S \\ v \notin S}}\omega_{uv}}{\sum_{u \in V(S)} q_u} \geq \epsilon\frac{\sum_{\substack{u \in S \\ v \notin S}}w_{uv}}{\sum_{u \in V(S)} q_u} = \epsilon h.
    \end{equation*}
\end{proof}

\subsection{Cheeger comparison theorems}
To obtain a somewhat useful comparison theorem, we require the following characterization of the weighted Cheeger constant $h$, which derives from a lengthy calculation beyond the scope of this paper. For a derivation that generalizes easily, see \cite{Chung2000}. Specifically, 
\begin{equation}\label{eqn:functional_h}
    h = \inf_{f\not\equiv 0}\sup_C \frac{\sum_{\{u,v\} \in E(G)} w_{uv}\phi_0(u)\phi_0(v)\abs{f(u)-f(v)}}{\sum_u q_u \phi_0^2(u) \abs{f(u)-C}}
\end{equation}
where $\phi_0$ is the ground-state of the corresponding Hamiltonian (Laplacian). Note that when $H$ is just a Laplacian, or $W=0$, $h$ is the standard Cheeger constant of the corresponding graph. With this, we can prove the following theorem:
\begin{thm}\label{thm:comparison}
    Suppose that $g$ is the Cheeger constant of $L_q$ corresponding to $G=(V,E)$ with weight function $w:V\times V \longrightarrow \mathbb{R}$. Further, suppose $h$ is the weighted Cheeger constant of $H=L_q+W$ resulting from imposing the Dirichlet condition as in \cref{sec:Dirichlet}. Then, if $H$ has ground-state $\phi$ satisfying the curvature inequality,
    \[
        \sum_{v\sim u} w_{uv}\abs{\phi(u)-\phi(v)} \leq \frac{\epsilon}{2} d_u \phi(u)
    \]
    the Cheeger constants $h$ and $g$ satisfy
    \[
        g \leq  h + \lambda_0(H) + \epsilon Q
    \]
    where $Q$ is the maximum degree of $G$ if $L_q$ is the combinatorial Laplacian and $1$ if $L_q$ is the normalized Laplacian.
\end{thm}
\begin{proof}
    Note that $g$ corresponds to a case where $\phi_0(u \in V(G)) = 1$ in \cref{eqn:functional_h}. Thus, 
    \[
        g = \inf_{f\not\equiv 0}\sup_C \frac{\sum_{\{u,v\} \in E(G)\cup \partial G} w_{uv}\abs{f(u)-f(v)}}{\sum_u q_u  \abs{f(u)-C}}.
    \]
    Now, let $S \subseteq G$ be the subset of $G$ that achieves $h$. We introduce
    \[  
        f(u)-C = \begin{cases}
            \phi^2(u) & u \in S \\
            - \phi^2(u) & u \notin S,
        \end{cases}
    \]
    where $\phi$ is the ground-state of $H$. Now,
    \begin{dgroup*}
    \[
        g \leq \frac{\displaystyle \sum_{\{u,v\} \in \partial S} w_{uv} \left(\phi^2(u) +\phi^2(v)\right) + \sum_{\{u,v\} \notin \partial S} w_{uv}\abs{\phi^2(u)-\phi^2(v)}}{\sum_u q_u \phi^2(u) }
    \]
    \[
        = \frac{\displaystyle \sum_{\{u,v\} \in \partial S} w_{uv} \left[\left(\phi(u) -\phi(v)\right)^2+2\phi(u)\phi(v) \right]+ \sum_{\{u,v\} \notin \partial S} w_{uv}\abs{\phi^2(u)-\phi^2(v)}}{\sum_u q_u \phi^2(u) }
    \]
    \[
        \leq 2h\frac{\sum_{u\in V(S)}q_u\phi^2(u) }{\sum_u q_u \phi^2(u)}+ \frac{\displaystyle \sum_{\{u,v\} \in \partial S} w_{uv} \left(\phi(u) -\phi(v)\right)^2+ \sum_{\{u,v\} \notin \partial S} w_{uv}\abs{\phi^2(u)-\phi^2(v)}}{\sum_u q_u \phi^2(u) }
    \]
    \[
        \leq h+ \frac{\displaystyle \sum_{\{u,v\} \in E(G)} w_{uv} \left(\phi(u) -\phi(v)\right)^2+ \sum_{\{u,v\} \notin \partial S} w_{uv}\left(\abs{\phi^2(u)-\phi^2(v)}-\left(\phi(u)-\phi(v)\right)^2\right)}{\sum_u q_u \phi^2(u) }
    \]    
    \[
        = h+ \lambda_0(H)  + \frac{\displaystyle \sum_{\{u,v\} \notin \partial S} w_{uv}\bigg[2 \min\{\phi(u),\phi(v)\} \;\abs{\phi(u)-\phi(v)}\bigg]}{\sum_u q_u \phi^2(u) }
    \]    
    \[
        \leq h+ \lambda_0(H) + \frac{\displaystyle \sum_{\{u,v\}} w_{uv}\bigg[2 \min\{\phi(u),\phi(v)\} \;\abs{\phi(u)-\phi(v)}\bigg]}{\sum_u q_u \phi^2(u) }
    \]    
    \[
        \leq h+ \lambda_0(H)  + 2\frac{\displaystyle \sum_{u}\phi(u)\sum_{v\sim u} w_{uv}\abs{\phi(u)-\phi(v)}}{\sum_u q_u \phi^2(u) }
    \]    
    \[
        \leq h + \lambda_0(H) + \epsilon\frac{\displaystyle \sum_{u} d_u \phi^2(u)}{\sum_u q_u \phi^2(u) }
    \]    
    \[
       \leq h + \lambda_0(H)  + \epsilon Q .
    \]
    \end{dgroup*}
    
    Thus,
    \[
        g\leq h + \lambda_0(H) + \epsilon Q.
    \]
\end{proof}
The above theorem is not as tight as we would ideally like. In the future, it would be advantageous to derive a better analogue of the results in \cite{cheng1997isoperimetric}. Although continuous, those results suggest that one could derive a comparison theorem such that $c h \geq g$ for some constant $c$ that depends only upon the structure of the space. Additionally, it seems likely that in the case that $\phi$ is unimodal, the weighted Cheeger constant is proportional to the Cheeger constant of the host graph. Nonetheless, a proof remains elusive. 

\subsubsection{Subgraph Comparison}
We can prove something a bit better by comparing subgraphs of our Hamiltonian and applying \cref{lem:potential}. For any $S$, let $h_S$ be as in \cref{eqn:local_Cheeger}. That is,
\begin{equation}\label{eqn:hs}
    h_S = \frac{\abs{\partial S}}{\min\{\vol(S),\vol(\overline{S})\}}
\end{equation}
where all quantities are as in \cref{sec:cheeger}.

If we again restrict to the case that $H$ is stoquastic, then we can apply the technique of \cref{lem:potential} to prove a theorem which makes clear the significance of the Cheeger constant for any particular cut $S \subset G$. In the following theorem, we make use of the Dirichlet representation of \cref{sec:Dirichlet}. In other words,
    \[
        \lambda_0(H) = \inf_{\substack{f \\ f\vert_{\delta G}=0}} \frac{\sum_{\{u,v\} \in E(G)\cup \partial G}w_{uv}(f(u)-f(v))^2}{\sum_{u \in V(G)} q_u f^2(u)}.
    \]
Thus, for any subgraph $S \subseteq G$, we can consider $\delta G \subseteq \delta S$. Another way of stating this, is that 
\begin{dgroup*}
    \[
        \lambda_0^D(H,S) = \inf_{\substack{f \\ f\vert_{\delta S}=0}} \frac{\sum_{\{u,v\} \in E(S)\cup \partial S}w_{uv}(f(u)-f(v))^2}{\sum_{u \in V(S)} q_u f^2(u)}
    \]
    \[
        = \inf_{\substack{f \\ f\vert_{\delta S}=0}} \frac{\sum_{\{u,v\} \in E(S)\cup (\partial S \setminus \partial G)}w_{uv}(f(u)-f(v))^2+\sum_{u \in V(S)} q_uW_u \phi^2(u)}{\sum_{u \in V(S)} q_u f^2(u)}.
    \]
\end{dgroup*}
The following theorem compares the Dirichlet eigenvalues of the subgraph $S$ to those of $G$. 

\begin{thm}
Suppose that $H$ is a stoquastic Hamiltonian with ground state $\phi>0$, corresponding to a graph $G$ with subgraph $S \subset G$. Then,
    \[
        h_S \geq \lambda_0^D(H,S) - \lambda_0(H).
    \]
    Above, $h_S$ is as in \cref{eqn:hs} and $\lambda_0^D(H,S)$ is the Dirichlet eigenvalue of the subgraph $S$ of the host graph $G\subseteq G'$, defined by
    \[
        \lambda_0^D(H,S) = \inf_{\substack{f \\ f\vert_{\delta S}=0}} \frac{\sum_{\{u,v\} \in E(S)\cup \partial S}w_{uv}(f(u)-f(v))^2}{\sum_{u \in V(S)} q_u f^2(u)}.
    \]
\end{thm}
\begin{proof}
    First, we begin with the appropriate definition of the Dirichlet eigenvalues of a subgraph.

    We begin as in \cref{lem:potential}. Without loss of generality, assume that $\vol(S) \leq \vol(\overline{S})$. Now, 
    \begin{dgroup*}
    \[
        \sum_{u \in V(S)}(\lambda_0(H)-W_u)q_u\phi^2(u) = \sum_{u\in S}\sum_{\{v,u\}\in E(G)}w_{uv}(\phi(u)-\phi(v))\phi(u)
    \]
    \[
        \lambda_0(H) \sum_{u \in V(S)}q_u\phi^2(u) = \sum_{\{u,v\} \in E(S)}w_{uv}(\phi(u)-\phi(v))^2 + \sum_{u \in V(S)}q_u W_u \phi^2(u) + \sum_{\{v,u\} \in \partial S \setminus \partial G}w_{uv}\left(\phi(u)-\phi(v)\right)\phi(u)
    \]
    \[
        \geq \lambda_0^D(H,S) \sum_{u \in V(S)}q_u \phi^2(u) - \sum_{\{v,u\} \in \partial S} w_{uv} \phi(v)\phi(u)
    \]
    \[
        = \left(\lambda_0^{D}(H,S) - h_S\right)\sum_{u \in V(S)}q_u\phi^2(u).
    \]
    \end{dgroup*}
    Since we know that $\sum_{u \in V(S)}q_u \phi^2(u) > 0$, 
    \[
        h_S \geq \lambda_0^D(H,S) - \lambda_0(H).
    \]
\end{proof}
In other words, whenever $h_S$ is exponentially small, there exists a Dirichlet eigenfunction for some subgraph that approximates the ground-state eigenvalue of $H$. This is equivalent to saying that there exists some block of $H$ that has approximately the same ground-state eigenvalue as $H$ itself.

\section{Physical implications}\label{sec:discussion}
These results lead to a very concrete understanding of the nature of the spectral gap in most quantum systems. In a very strong sense, the presence of a spectral gap implies that the ground-state wave function \textit{must not} contain bottlenecks. Although this may be unsurprising, all prior results fail to confirm the intuition when $\norm{W}$ is sufficiently large. In this paper, I have eliminated the ability for physics behave unexpectedly in such situations. That is, we now know that gapped Hamiltonians must not contain strong bottlenecks in their ground-states and, additionally, the appropriate scaling of this claim. Equivalently, the presence of a bottleneck guarantees a small spectral gap.

This conceptual point does not yet hold in reverse. That is, we have not shown that a small gap implies a strong bottleneck. It is possible that there exist Hamiltonians with ground-states without bottlenecks that nonetheless have small spectral gaps. This particular point may be of some physical interest and worth exploring, however in the context that inspired this work is somewhat less interesting.

Probably the major advantage of this characterization is that we can now definitively say that, for the standard adiabatic theorem to guarantee an efficient adiabatic process, at no point in the evolution must $H$ have a bottlenecked ground-state. Some results suggest that, at least with existing Monte Carlo techniques, states without bottlenecks can still be hard to simulate \cite{Hastings,jarret2016adiabatic,bringewatt2018diffusion}. Nonetheless, a guaranteed lack of bottlenecks reaffirms my agnosticism about whether one might be able to classically and efficiently sample from ground-state distributions arising from large-gap stoquastic Hamiltonians. Shifting dialogue away from spectral gaps and towards bottlenecked distributions as also suggested in \cite{Jarret2014a,crosson2017quantum} will, hopefully, shed light on this question one way or the other.

\section{The Bashful Adiabatic Algorithm}
In this section, I show how one might be able to exploit the weighted Cheeger constant to improve quantum adiabatic algorithms. A quantum process solves the Schr\"{o}dinger equation
\[  \begin{cases}
        i \frac{\partial \phi(t)}{\partial t} = H(t/T) \phi(t) \\
        \phi(0) = \phi_0(0)
    \end{cases}
\]
where $\phi_0(t)$ is the ground-state of $H(t/T)$.  An adiabatic algorithm seeks to produce the distribution $\phi(T) \approx \phi_0(T)$ and the adiabatic theorem guarantees that this can be done provided that a quantity like $\gamma^{-2}(H(t/T))\norm{\frac{dH(t/T)}{dt}}$ is never too large \cite{Jansen2006}. Abusively, for this section, we call the Hamiltonian $H(t/T)$ the ``schedule''. At least in the case of real Hamiltonians, our inequality opens up the possibility of adaptive adiabatic algorithms, or those where we adjust the rate of variation of $H$ in response to the size of the gap.

In many cases, $h$ reduces the problem of bounding the spectral gap to determining information about the ground-state. This allows one to stop an evolution early, say at $t < T$ and bound the gap at that point. That is, suppose we know $\phi(t) \approx \phi_0(t)$ for some $t$. Then, if we can use $\phi(t)$ to approximate $h$, we can assume that we know $\gamma(H(t/T))$. One can use Weyl's inequality or another perturbative argument to then guarantee that $\gamma(H(\tau/T)) \geq c$ for some choice of $c$ and $\tau > t$. Thus, we can restart the adiabatic from $t=0$ and choose an appropriate $dH(t/T)/dt$ such that $\phi(\tau) \approx \phi_0(\tau)$. Repeating this until $\tau = T$ would give us the entire adiabatic path with, potentially, only polynomial overhead. This algorithm, which I am calling the Bashful Adiabatic Algorithm (BAA), is sketched below:\footnote{BAA reminds me of its sheepishness.}

\begin{algorithm}[H]
\caption*{\textbf{Bashful Adiabatic Algorithm}}
\begin{algorithmic}[1]\label{alg:find_eta}
    \State Assume $H_\tau(1) = H_0(1)$ for all choices of $\tau$.
    \State Choose a schedule $H_{\tau}$ with $\min_{t < \tau}\gamma(H_{\tau}(t/T)) > \gamma_\min$.
    \State Prepare the state $\phi_0(0)$ of $H_0(0)$. 
    \While{$\tau < T$}  
        \State Generate $N$ copies of $\phi(\tau)$ from $\phi(0)$ using the schedule $H_\tau(t/T)$.
        \State Sample $\{\phi(\tau)\}$ and (if possible) approximate the weighted Cheeger constant of $H_{\tau}(\tau)$.
        \State Use the result to bound $\min_{t <\tau + \delta \tau}\gamma(H_{\tau+\delta \tau}(t/T))$ for some new schedule $H_{\tau + \delta \tau}(t/T)$. 
        \State $\tau \gets \tau + \delta \tau$.
    \EndWhile
    \Return $\phi(T)$ using the schedule $H_T(t/T)$.
\end{algorithmic}
\end{algorithm}
This algorithm would run in time $\bigO{(T/\delta \tau)^2(X+N\delta \tau)}$, where $\delta \tau$ is the smallest timestep taken, $N$ is the number of copies needed, and $X$ the longest time it takes to compute $h$. The reader should note that even if $\delta \tau$ must get very small (because $\gamma$ gets very small), so long as it is only small for a sufficiently short period of time, we should be able to locally decrease $\norm{\frac{dH}{dt}}$ and obtain much tighter scaling than that proposed above. Furthermore, we can ensure that our $\norm{\frac{d H}{dt}}$ is taken as large as possible while remaining consistent with the adiabatic theorem, or that our path (through time) is chosen optimally. The ability to compute $h$ may allow one to predict when an adiabatic path needs to be changed, as suggested in \cite{crosson2014different}. 

Even given the ability to sample $\phi_0$, we would still require an efficient method for approximating $h$. Although I do not expect this to be possible for an arbitrary graph and $\phi_0$, this may indeed be possible for some classes of graphs and reasonable assumptions about $\phi_0$. It is likely that a statement like \cref{lem:potential} will be useful in this regard. Additionally, while there will clearly be distributions where an approximation strategy for $h$ should fail, it is quite possible that these same instances correspond to otherwise intractable optimization problems. 

As an example, one can think of the graph $G=(V,E)$ with $V = \{u_i \; \vert \; i \in \intrange{1}{n} \}$ and $E = \{\{u_i,u_{i+1} \} \; \vert \; i \in \intrange{1}{n-1}\}$. Suppose that for some $j \notin \{i,i+1\}$, the Hamiltonian has ground-state
\[
    \phi_0(u_i,\tau) = \begin{cases}
        c_1 & i = 1 \\
        c_j & i = j \\
        C & i \notin \{1,j\} .
    \end{cases}
\]
Choosing $C \sim e^{-n}$, if $c_1 > c_j \sim \mathrm{poly}(n)$, then there exists a cut such that $h$ is exponentially small in $n$. Using $L$ as the graph Laplacian for this graph, this is achieved by, for example, the ground-state of $H=L+W$ with diagonal matrix $W \equiv \diag{(W_u)_{u \in V}}$ 
\[
    W_{u_i} = \begin{cases}
            c x^{-1} & i=1\\    
            x c^{-1} & i = 2\\ 
            c^{-1} & i = \abs{V}-1\\
            c & i = \abs{V}\\
            1 & \text{otherwise}
    \end{cases}
\]
and an appropriate choice of $c$ and $x$. (Take $c$ to be small and choose $x$ to produce the desired ratio of $c_1/c_{\abs{V}}$.)

Distinguishing this from the case where $c_j \sim e^{-n}$, which implies that $h$ is only polynomially small in $n$ (see \cite{Jarret2014a}), seems to be close to efficiently solving unstructured search. Thus, if one were to investigate an algorithm for approximating $h$, one might need to consider a divide-and-conquer approach that considers separate  adiabatic processes constrained to different subgraphs for sufficiently concentrated $\phi_0$. Another possibility would be to attempt to adapt existing algorithms for approximating the Cheeger constant in large networks \cite{spielman2004nearly}. Exploring this question is well beyond the scope of the present work, but would nonetheless be very interesting.

\section{Open questions and future work}
These inequalities lead to quite a few open questions. 
\begin{itemize}
\item First and foremost, I think, is the question of whether one can ever efficiently approximate the weighted Cheeger constant and what information/constraints would be necessary to do so. The standard combinatorial Cheeger constant has been the object of extensive study and we know determining it to be NP-hard \cite{matula1990sparsest}. Nonetheless, one can efficiently approximate the Cheeger constant, however the scaling of such estimates is probably insufficient for quantum systems. Additionally, given that the weighted Cheeger constant depends on more information than the combinatorial Cheeger constant, estimating the weighted Cheeger constant might be considerably harder. Nonetheless, it is possible that in sparse graphs, such as those that would naturally arise from physical systems of interest, this quantity might not be too difficult to approximate, especially if one is willing to take a poor estimate. If one can approximate $h$ efficiently enough in a large enough number of cases, one might potentially use this information to choose an adiabatic path for adiabatic quantum computation as discussed in the previous section \cite{crosson2014different}.

\item Also, because this work demonstrates the deficiencies in gap analysis, it would be interesting if one could prove a version of the adiabatic theorem specific to bottlenecked states. In particular, an adiabatic theorem that stresses Dirichlet eigenfunctions would probably be able to capture the ``relevant'' portion of the wavefunction. One can imagine a situation where the solution to some optimization problem is in a subgraph $S\subseteq G$ where there exists no bottleneck and $\phi$ is large and, yet, $\overline{S}$ contains a strong bottleneck somewhere. It would be interesting to see if such situations arise frequently, infrequently, or never. I suspect they arise frequently, and thus deriving adiabatic theorems that restrict to the subgraph $S$ that we wish to explore would have a hope of providing much better runtime bounds.

\item Another question is whether one can derive useful comparison theorems between the gap of the host graph and the gap of Hamiltonian, as  alluded to in \cref{sec:applications}. Desirable forms for comparison theorems can be found in many places, such as \cite{Chung2000,Chung}. (The interested reader should beware, however, as \cite[Theorem 3]{Chung2000} is incorrect due to a sign error and the result is carried through to two of the main corollaries of the paper. Theorem 4 of that paper also appears to be incorrect, and the best one can hope for is a statement like the present \cref{thm:comparison}.) It seems likely that, at least for unimodal ground-states on strongly convex subgraphs of homogeneous graphs (see \cite{Chung}), one should be able to show that the gap of the Hamiltonian scales with the gap of the graph. Additionally, \cite{Jarret2014a} shows that a condition like log-concavity is not enough to guarantee unimodality. In that paper, a seemingly bimodal distribution can satisfy log-concavity due to the nature of the boundary, whereas the continuous definition of log-concavity would imply unimodality.

\item Finally, one might consider what useful information the frustration index can provide about the spectral gap. In \cite{Martin2017}, the author derives isoperimetric inequalities that utilize the frustration index. It is entirely possible that a suitably defined index can yield tighter bounds than those derived through our reductions here. It also seems likely that this concept might be a key component to obtaining gap lower bounds in the general Hermitian case.

\end{itemize}
\section{Acknowledgements}
The idea for using $h$ to adjust the adiabatic path was arrived at during exchanges with Antonio Martinez. Kianna Wan pointed out many small errors that would have otherwise went unnoticed, helping me greatly improve my presentation. I thank Elizabeth Crosson, Stephen Jordan, Tsz Chiu Kwok, Brad Lackey, Lap Chi Lau, and Adrian Lupascu for helpful discussions. \PIRA

\appendix
\section{Proof of \cref{thm:nonstoq}}
\label{ap:proof}

First, we note that in the proof of \cref{thm:cheeger}, we had the following corollary. 
\begin{cor}\label{cor:nonstoq}
    For a graph $G=(V,E)$, suppose
    \[
        \gamma = \inf_{\tiny{g \perp q\phi^2}} \frac{\sum_{\{u,v\} \in E(G^+)}\widetilde\omega_{uv}[g(u)-g(v)]^2}{\sum_u q_u g^2(u)\phi^2(u)}.
    \]
    Then, for $f$ achieving the infimum above and 
    \[
        g(u) = \begin{cases}
        f(u) & f(u)\geq 0 \\
        0 & \text{otherwise},
        \end{cases}
    \]
    we have
    \[
        {\gamma \geq \frac{\displaystyle\sum_{\{v,u\}}\omega_{uv}[g(u)-g(v)]^2}{\displaystyle\sum_{u}q_u g^2(u)\phi_0^2(u)} \geq \frac{\left(\displaystyle\sum_{\{v,u\}} \omega_{uv}\abs{g^2(u)-g^2(v)}\right)^2}{\left(\displaystyle\sum_{u\in V(S)}q_uf^2(u)\phi_0^2(u)\right) \left(\displaystyle\sum_{\{v,u\}}\omega_{uv}[g(u)+g(v)]^2\right)} }.
    \]
\end{cor}
Now, we can prove \cref{thm:nonstoq} by
adapting the proof of \cref{thm:cheeger}. First, we note that by \cref{cor:nonstoq},
\begin{dgroup*}
\[
        {\gamma \geq  \Phi} = \frac{\displaystyle\sum_{\{v,u\}}\omega_{uv}[g(u)-g(v)]^2}{\displaystyle\sum_{u}q_u g^2(u)\phi_0^2(u)} \geq \frac{\left(\displaystyle\sum_{\{v,u\}} \omega_{uv}\abs{g^2(u)-g^2(v)}\right)^2}{\left(\displaystyle\sum_{u\in V(S)}q_uf^2(u)\phi_0^2(u)\right) \left(\displaystyle\sum_{\{v,u\}}\omega_{uv}[g(u)+g(v)]^2\right)}
     \]\[ =\frac{\left(\displaystyle\sum_{\{v,u\}} \omega_{uv}\abs{g^2(u)-g^2(v)}\right)^2}{\left(\displaystyle\sum_{u\in V(S)}q_uf^2(u)\phi_0^2(u)\right) \left(\displaystyle 2\sum_{\{v,u\}}\omega_{uv}[g^2(u)+g^2(v)] -\sum_{\{v,u\}}\omega_{uv}[g(u)-g(v)]^2\right)}
     \]\[
        = \frac{\left(\displaystyle\sum_{\{v,u\}} \omega_{uv}\abs{g^2(u)-g^2(v)}\right)^2}{\left(\displaystyle\sum_{u\in V(S)}q_uf^2(u)\phi_0^2(u)\right) \left(\displaystyle 2\sum_{u}g^2(u) \sum_{v\sim u}\omega_{uv} -\sum_{\{v,u\}}\omega_{uv}[g(u)-g(v)]^2\right)}
      \]\[ \geq\frac{\left(\displaystyle\sum_{\{v,u\}} \omega_{uv}\abs{g^2(u)-g^2(v)}\right)^2}{\left(\displaystyle\sum_{u\in V(S)}q_uf^2(u)\phi_0^2(u)\right) \left(\displaystyle 2\sum_{u}g^2(u) \sum_{v\sim u}\phi(u)\phi(v)w_{uv} -\sum_{\{v,u\}}\omega_{uv}[g(u)-g(v)]^2\right)}
       \]\[ \geq\frac{\left(\displaystyle\sum_{\{v,u\}} \omega_{uv}\abs{g^2(u)-g^2(v)}\right)^2}{\left(\displaystyle\sum_{u\in V(S)}q_uf^2(u)\phi_0^2(u)\right) \left(\displaystyle 2\sum_{u}q_u f^2(u)\phi^2(u)\left(W_u + \frac{d_u}{q_u}- \lambda_0\right) -\sum_{\{v,u\}}\omega_{uv}[g(u)-g(v)]^2\right)}
        \geq\frac{\left(\displaystyle\sum_{\{v,u\}} \omega_{uv}\abs{g^2(u)-g^2(v)}\right)^2}{\left(\displaystyle\sum_{u\in V(S)}q_uf^2(u)\phi_0^2(u)\right)^2 \left(\displaystyle 2Q + 2 \left(\lambda_{\abs{V}-1}-\lambda_0\right) - \Phi \right)}
       \geq\frac{\left(\displaystyle\sum_{\{v,u\}} \omega_{uv}\abs{g^2(u)-g^2(v)}\right)^2}{\left(\displaystyle\sum_{u\in V(S)}q_uf^2(u)\phi_0^2(u)\right)^2 \left(\displaystyle 2Q + 2\rho - \Phi \right)}.
    \]
\end{dgroup*}
    The remainder of this proof follows identically the remaining portion of the proof of \cref{thm:cheeger}.

\bibliographystyle{plain}

\begin{thebibliography}{10}

\bibitem{Jarret2014}
{Adiabatic Optimization and Dirichlet Graph Spectra}.
\newblock Quantum Optimization Workshop, Fields Institute, Toronto, ON, Canada,
  8 2014.

\bibitem{al2010energy}
Abbas Al-Shimary and Jiannis~K Pachos.
\newblock Energy gaps of hamiltonians from graph laplacians.
\newblock {\em arXiv preprint arXiv:1010.4130}, 2010.

\bibitem{Albash2018}
Tameem Albash and Daniel~A. Lidar.
\newblock {Adiabatic quantum computation}.
\newblock {\em Reviews of Modern Physics}, 90(1), 11 2018.

\bibitem{andrews2011proof}
Ben Andrews and Julie Clutterbuck.
\newblock Proof of the fundamental gap conjecture.
\newblock {\em Journal of the American Mathematical Society}, 24(3):899--916,
  2011.

\bibitem{Atay2014}
Fatihcan~M. Atay and Shiping Liu.
\newblock {Cheeger constants, structural balance, and spectral clustering
  analysis for signed graphs}.
\newblock {\em Citeseer}, 2014.

\bibitem{atay2014spectrum}
Fatihcan~M Atay and Hande Tuncel.
\newblock On the spectrum of the normalized laplacian for signed graphs:
  Interlacing, contraction, and replication.
\newblock {\em Linear Algebra and its Applications}, 442:165--177, 2014.

\bibitem{barahona}
F~Barahona.
\newblock {On the computational complexity of Ising spin glass models}.
\newblock {\em Journal of Physics A: Mathematical and General}, 15(10):3241,
  1982.

\bibitem{Bauer2012}
Frank Bauer.
\newblock {Normalized graph Laplacians for directed graphs}.
\newblock {\em Linear Algebra and Its Applications}, 436(11):4193--4222, 2012.

\bibitem{bravyi2008complexity}
Sergey Bravyi, David~P Divincenzo, Roberto Oliveira, and Barbara~M Terhal.
\newblock The complexity of stoquastic local hamiltonian problems.
\newblock {\em Quantum Information \& Computation}, 8(5):361--385, 2008.

\bibitem{bringewatt2018diffusion}
Jacob Bringewatt, William Dorland, Stephen~P Jordan, and Alan Mink.
\newblock Diffusion monte carlo approach versus adiabatic computation for local
  hamiltonians.
\newblock {\em Physical Review A}, 97(2):022323, 2018.

\bibitem{Chan2015}
T.~H~Hubert Chan, Zhihao~Gavin Tang, and Chenzi Zhang.
\newblock {Cheeger inequalities for general edge-weighted directed graphs}.
\newblock {\em Lecture Notes in Computer Science (including subseries Lecture
  Notes in Artificial Intelligence and Lecture Notes in Bioinformatics)},
  9198:30--41, 2015.

\bibitem{cheng1997isoperimetric}
Shiu-Yuen Cheng and Kevin Oden.
\newblock Isoperimetric inequalities and the gap between the first and second
  eigenvalues of an euclidean domain.
\newblock {\em The Journal of Geometric Analysis}, 7(2):217--239, 1997.

\bibitem{Chung}
F~R~K Chung.
\newblock {\em {Spectral Graph Theory}}, volume~92 of {\em CBMS Regional
  Conference Series in Mathematics}.
\newblock American Mathematical Society, Providence, Rhode Island, 12 1997.

\bibitem{Chung2005}
Fan Chung.
\newblock {Laplacians and the Cheeger inequality for directed graphs}.
\newblock {\em Annals of Combinatorics}, 9(1):1--19, 2005.

\bibitem{Chung2000}
Fan R~K Chung and Kevin Oden.
\newblock {Weighted graph Laplacians and isoperimetric inequalities}.
\newblock {\em Pacific Journal of Mathematics}, 192(2):257--273, 2000.

\bibitem{cloez2016fleming}
Bertrand Cloez and Marie-No{\'e}mie Thai.
\newblock Fleming-viot processes: two explicit examples.
\newblock 2016.

\bibitem{cloez2016quantitative}
Bertrand Cloez and Marie-No{\'{e}}mie Thai.
\newblock Quantitative results for the fleming--viot particle system and
  quasi-stationary distributions in discrete space.
\newblock {\em Stochastic Processes and their Applications}, 126(3):680--702,
  2016.

\bibitem{collet2012quasi}
Pierre Collet, Servet Mart{\'i}nez, and Jaime San~Mart{\'i}n.
\newblock {\em Quasi-stationary distributions: Markov chains, diffusions and
  dynamical systems}.
\newblock Springer Science \& Business Media, 2012.

\bibitem{collet2013markov}
Pierre Collet, Servet Mart{\'i}nez, and Jaime San~Mart{\'i}n.
\newblock Markov chains on finite spaces.
\newblock In {\em Quasi-Stationary Distributions}, pages 31--44. Springer,
  2013.

\bibitem{crosson2017quantum}
Elizabeth Crosson and John Bowen.
\newblock Quantum ground state isoperimetric inequalities for the energy
  spectrum of local hamiltonians.
\newblock 2017.

\bibitem{crosson2014different}
Elizabeth Crosson, Edward Farhi, Cedric Yen-Yu Lin, Han-Hsuan Lin, and Peter
  Shor.
\newblock Different strategies for optimization using the quantum adiabatic
  algorithm.
\newblock {\em arXiv preprint arXiv:1401.7320}, 2014.

\bibitem{Crosson2016}
Elizabeth Crosson and Aram~W. Harrow.
\newblock {Simulated Quantum Annealing Can Be Exponentially Faster Than
  Classical Simulated Annealing}.
\newblock In {\em 2016 IEEE 57th Annual Symposium on Foundations of Computer
  Science (FOCS)}, pages 714--723. IEEE, 8 2016.

\bibitem{diaconis1991geometric}
Persi Diaconis and Daniel Stroock.
\newblock Geometric bounds for eigenvalues of markov chains.
\newblock {\em The Annals of Applied Probability}, pages 36--61, 1991.

\bibitem{GAREY1976237}
M~R Garey, D~S Johnson, and L~Stockmeyer.
\newblock {Some simplified NP-complete graph problems}.
\newblock {\em Theoretical Computer Science}, 1(3):237--267, 1976.

\bibitem{HARARY1980131}
Frank Harary and Jerald~A Kabell.
\newblock {A simple algorithm to detect balance in signed graphs}.
\newblock {\em Mathematical Social Sciences}, 1(1):131--136, 1980.

\bibitem{Hastings}
M.~B. Hastings.
\newblock Obstructions to classically simulating the quantum adiabatic
  algorithm.
\newblock {\em Quantum Information and Computation}, 13(11/12):1038--1076,
  2013.
\newblock With appendix by M. H. Freedman.

\bibitem{Jansen2006}
Sabine Jansen, Mary-Beth Ruskai, and Ruedi Seiler.
\newblock {Bounds for the adiabatic approximation with applications to quantum
  computation}.
\newblock {\em Journal of Mathematical Physics}, 102111(2007):15, 2006.

\bibitem{Jarret2016}
M~Jarret, S~P Jordan, and B~Lackey.
\newblock {Adiabatic optimization versus diffusion Monte Carlo methods}.
\newblock {\em Physical Review A - Atomic, Molecular, and Optical Physics},
  2016.

\bibitem{Jarret2014a}
Michael Jarret and Stephen~P Jordan.
\newblock Adiabatic optimization without local minima.
\newblock {\em Quantum Information \& Computation}, 14(Quantum Information \&
  Computation), 2015.

\bibitem{jarret2016adiabatic}
Michael Jarret, Stephen~P Jordan, and Brad Lackey.
\newblock Adiabatic optimization versus diffusion monte carlo methods.
\newblock {\em Physical Review A}, 94(4):042318, 2016.

\bibitem{jarret2017substochastic}
Michael Jarret and Brad Lackey.
\newblock Substochastic monte carlo algorithms.
\newblock 2017.

\bibitem{kaibel2004expansion}
Volker Kaibel.
\newblock On the expansion of graphs of 0/1-polytopes.
\newblock In {\em The Sharpest Cut: The Impact of Manfred Padberg and His
  Work}, pages 199--216. SIAM, 2004.

\bibitem{kannan2004clusterings}
Ravi Kannan, Santosh Vempala, and Adrian Vetta.
\newblock On clusterings: Good, bad and spectral.
\newblock {\em Journal of the ACM (JACM)}, 51(3):497--515, 2004.

\bibitem{Lange2015}
Carsten Lange, Shiping Liu, Norbert Peyerimhoff, and Olaf Post.
\newblock {Frustration index and Cheeger inequalities for discrete and
  continuous magnetic Laplacians}.
\newblock {\em Calculus of Variations and Partial Differential Equations},
  54(4):4165--4196, 12 2015.

\bibitem{leighton1988approximate}
Tom Leighton and Satish Rao.
\newblock An approximate max-flow min-cut theorem for uniform multicommodity
  flow problems with applications to approximation algorithms.
\newblock In {\em Foundations of Computer Science, 1988., 29th Annual Symposium
  on}, pages 422--431. IEEE, 1988.

\bibitem{Martin2017}
Florian Martin.
\newblock {Frustration and isoperimetric inequalities for signed graphs}.
\newblock {\em Discrete Applied Mathematics}, 217:276--285, 1 2017.

\bibitem{Marvian2018}
Milad Marvian, Daniel~A. Lidar, and Itay Hen.
\newblock {On the Computational Complexity of Curing the Sign Problem}.
\newblock pages 1--12, 2018.

\bibitem{matula1990sparsest}
David~W Matula and Farhad Shahrokhi.
\newblock Sparsest cuts and bottlenecks in graphs.
\newblock {\em Discrete Applied Mathematics}, 27(1-2):113--123, 1990.

\bibitem{roland2002quantum}
J{\'e}r{\'e}mie Roland and Nicolas~J Cerf.
\newblock Quantum search by local adiabatic evolution.
\newblock {\em Physical Review A}, 65(4):042308, 2002.

\bibitem{sachdev2007quantum}
Subir Sachdev.
\newblock {\em Quantum phase transitions}.
\newblock Wiley Online Library, 2007.

\bibitem{sher}
David Sherrington and Scott Kirkpatrick.
\newblock Solvable model of a spin-glass.
\newblock {\em Phys. Rev. Lett.}, 35:1792--1796, 12 1975.

\bibitem{sinclair2012algorithms}
Alistair Sinclair.
\newblock {\em Algorithms for random generation and counting: a Markov chain
  approach}.
\newblock Springer Science \& Business Media, 2012.

\bibitem{spielman2004nearly}
Daniel~A Spielman and Shang-Hua Teng.
\newblock Nearly-linear time algorithms for graph partitioning, graph
  sparsification, and solving linear systems.
\newblock In {\em Proceedings of the thirty-sixth annual ACM symposium on
  Theory of computing}, pages 81--90. ACM, 2004.

\bibitem{troyer2005computational}
Matthias Troyer and Uwe-Jens Wiese.
\newblock Computational complexity and fundamental limitations to fermionic
  quantum monte carlo simulations.
\newblock {\em Physical review letters}, 94(17):170201, 2005.

\bibitem{yau2009estimate}
Shing-Tung Yau.
\newblock An estimate of the gap of the first two eigenvalues in the {S}chr\"{
  o}dinger operator.
\newblock 2009.

\end{thebibliography}

\end{document}